\documentclass[11pt]{article}

\usepackage{amsthm}
\usepackage{amsmath}
\usepackage{amssymb}
\usepackage{graphicx}
\usepackage{xcolor}
\usepackage{array}
\usepackage{multirow}
\usepackage{caption}
\usepackage{enumitem}
\usepackage{natbib}
\usepackage{float}
\usepackage{tikz}
\usepackage{pgfplots}
\usepackage{fontenc}
\usepackage{amsthm}
\usepackage{multirow}

\usepackage[margin=1.25truein]{geometry}

\renewcommand\baselinestretch{1.1}

\usetikzlibrary{arrows,calc,angles,positioning,intersections,quotes,decorations.markings,backgrounds,patterns}
\usetikzlibrary{decorations.pathreplacing}
\pgfplotsset{compat=1.11}
\tikzset{
	mynode/.style={fill,circle,inner sep=1pt,outer sep=0pt}
}

\newcommand{\mbf}[1]{\mbox{\boldmath $#1$}}
\newcommand{\ba}{{\mbf \beta}}

{\catcode `\@=11 \global\let\AddToReset=\@addtoreset}
\AddToReset{equation}{section}
\renewcommand{\theequation}{\thesection.\arabic{equation}}

\AddToReset{Theorem}{section}

\newtheorem{cor}{Corollary}[section]
\newtheorem{lem}{Lemma}[section]
\newtheorem{rem}{Remark}[section]
\newtheorem{thm}{Theorem}[section]

\newcommand{\cA}{{\cal A}}
\newcommand{\cB}{{\cal B}}

\newcommand{\cG}{{\cal G}}

\newcommand{\cK}{{\cal K}}

\newcommand{\cM}{{\cal M}}
\newcommand{\cN}{{\cal N}}

\newcommand{\cU}{{\cal U}}

\newcommand{\cW}{{\cal W}}

\def\ba{\begin{array}}
	\def\bc{\begin{center}}
		\def\bd{\begin{description}}
			\def\be{\begin{enumerate}}
				\def\ea{\end{array}}
			\def\ec{\end{center}}
		\def\ed{\end{description}}
	\def\edt{\end{document}}
\def\ee{\end{enumerate}}
\def\ben{\begin{equation}}
\def\benn{\begin{equation*}}
\def\een{\end{equation}}
\def\eenn{\end{equation*}}
\def\benr{\begin{eqnarray}}
\def\eenr{\end{eqnarray}}
\def\benrr{\begin{eqnarray*}}
\def\eenrr{\end{eqnarray*}}

\def\al{\alpha}

\def\edt{\end{document}}

\def\g{\gamma}
\def\G{\Gamma}
\def\h{\hat}
\def\ka{\kappa}

\def\iny{\infty}
\def\ka{\kappa}

\def\la{\lambda}

\def\noi{\noindent}

\def\nn{\nonumber}

\def\si{\sigma}

\def\Si{\Sigma}

\def\vep{\varepsilon}

\def\vs{\vskip}

\def\R{{\mathbb R}}

\DeclareMathOperator*{\argmin}{arg\,min}

\begin{document}

\bc
{\Large {\bf Inference on the change point with the jump size near the boundary of the region of detectability in high dimensional time series models}}\\[.2cm]
Abhishek Kaul$^{a,}$ Venkata K. Jandhyala$^a$, Stergios B. Fotopoulos$^b$


$^a$Department of Mathematics and Statistics,\\ $^b$Department of Finance and Management Science,\\ 
Washington State University, Pullman, WA 99164, USA.


\ec
\vs .1in
{\renewcommand{\baselinestretch}{1}
	\begin{abstract}
We develop a projected least squares estimator for the change point parameter in a high dimensional time series model with a potential change point. Importantly we work under the setup where the jump size may be near the boundary of the region of detectability. The proposed methodology yields an optimal rate of convergence despite high dimensionality of the assumed model and a potentially diminishing jump size. The limiting distribution of this estimate is derived, thereby allowing construction of a confidence interval for the location of the change point. A secondary near optimal estimate is proposed which is required for the implementation of the optimal projected least squares estimate. The prestep estimation procedure is designed to also agnostically detect the case where no change point exists, thereby removing the need to pretest for the existence of a change point for the implementation of the inference methodology. Our results are presented under a general positive definite spatial dependence setup, assuming no special structure on this dependence. The proposed methodology is designed to be highly scalable, and applicable to very large data. Theoretical results regarding detection and estimation consistency and the limiting distribution are numerically supported via monte carlo simulations.
\end{abstract} }
\noi {\it Keywords: High dimensions, time series, change point, inference, limiting distribution, region of detectability.}

\section{Introduction}\label{sec:intro}

In many applications of current scientific interest the assumption of stationarity of the mean of a time series over an extended sampling period could be unrealistic and may lead to flawed inference. Dynamic time series characterized via mean changes across unknown change points form a simplistic yet useful tool to model such non-stationarity of large streams of data. With large amounts of data now being commonplace in a variety of scientific fields such as econometrics, finance and genomics, significant attention in the statistical literature is being paid for the estimation of change points in a high dimensional setting, where the dimension of the time series being observed may be diverging much faster than the number of observations. In this article we consider the simplest of change point models, characterized as a linear process with a single potential mean shift, i.e.,
\benr\label{model:subgseries}
y_t=\begin{cases}\mu^0_1+\vep_t, & t=1,...,\lfloor T\tau^0\rfloor\\
	\mu^0_2+\vep_t,& t=\lfloor T\tau^0\rfloor+1,...,T.\end{cases}
\eenr
Here $\vep_t\in\R^p,$ $t=1,...,T$ are the unobserved noise random variables, which are assumed to be independent and identically distributed (i.i.d.) realizations of a $p$-dimensional zero mean subgaussian distribution\footnote{Recall that for $\al>0,$ the random variable $\eta$ is said to be $\al$-subgaussian if, for all $t\in\R,$ $E[\exp(t\eta)] \le \exp(\al^2t^2/2).$ Similarly, a random vector $\xi\in\R^p$ is said to be $\al$-subgaussian if the inner products $\langle\xi, v\rangle$ are $\al$-subgaussian for any $v\in\R^p$ with $\|v\|_2 = 1.$}. The observed variable is $y_t\in\R^p,$ and the unknown parameters are the means $\mu^0_1,\mu^0_2\in\R^p,$ and the change point parameter $\tau^0\in (0,1],$ with the latter being of main interest in this article. Note that, the case of `no change' is allowed by the model (\ref{model:subgseries}), since we allow $\tau^0=1,$ in its parametric space. In this case, model (\ref{model:subgseries}) reduces to $T$ observations of a stationary mean subgaussian distribution. Finally, we allow the dimension $p$ to diverge potentially at an exponential rate, i.e., $\log p=o(T^{\delta}),$ for some $0<\delta<1/2,$ while making a sparsity assumption to be specified in the following section.

The two main inferential problems of interest on $\tau^0$ of model (\ref{model:subgseries}) are, (a) whether a change point exists, i.e., test for the null hypothesis $H_0:\,\tau^0=1,$ and (b) construction of a confidence interval for the parameter $\tau^0$ when it exists, i.e., when $\tau^0<1.$ Despite the simplicity of model (\ref{model:subgseries}), the current literature discussing these inferential problems in the high dimensional setup is very sparse. Infact, in this high dimensional setting, solutions are available largely for problem (a), i.e., for the detection of a change point, see for e.g. \cite{enikeeva2013high}, \cite{wang2019inference}, \cite{li2019change} and \cite{steland2018inference} among others. In context of problem (b), the articles of \cite{bai2010common} and \cite{bhattacharjee2019change} consider the same linear single shift process as considered in this article. They develop inferential results using the ordinary least squares estimator applied directly on the $p$-dimensional model (\ref{model:subgseries}). The work of \cite{bai2010common} allows the dimension $p$ to diverge at an  arbitrarily rate with $T.$ The cost of such generality is paid by assuming a very large jump size $\xi=\|\mu_1^0-\mu_2^0\|_2,$ wherein the article assumes a diverging jump size satisfying $\xi/\surd{p}\to \iny,$ in order to obtain $T$-consistency of the estimate. The article of \cite{bhattacharjee2019change} considers a similar least square estimator, and assumes the jump size to satisfy $\xi\surd(T/p)\to \iny.$ While this assumption allows a diminishing jump size, however it does so only in the low dimensional case where $p/T\to 0.$ In the high dimensional setting, this condition again is only satisfied under a diverging jump. These two articles together illustrate the fact that either very large jump sizes, or low dimensions may be required in order to perform inference on the change point when the estimate is extracted from a high dimensional data set, without using any sparsity assumptions. On the other hand, it has also recently been shown in \cite{liu2019minimax} that assuming sparsity of the jump vector, much weaker signals in the jump size are detectable. Specifically, they show that the region of detectability (ROD) of the change point satisfies a  minimax rate of $\xi^{-1}\surd\big\{s\log (p\vee T)/T\big\}\le c,$ upto other logarithmic terms in $s$ and $T,$ under restrictions on the sparsity parameter $s.$ We refer to their article for the precise minimax rate which involves a tripe iterated log expression. In this more realistic high dimensional setup where the jump size is not arbitrarily large, \cite{wang2018high} provide a sparse projection estimator that yields a near optimal rate of convergence $\{\log(\log T)\}/T.$ To the best of our knowledge, this is at present the sharpest result regarding the rate of convergence of a change point estimate available in the literature, under high dimensionality, without a diverging jump assumption. In this setting, there is currently no available estimator of the change point $\tau^0$ that yields an optimal rate of convergence ($1/T$). Consequently there are no available limiting distribution results or methods to construct confidence regions for $\tau^0.$ The overarching objective of this article is to propose a sufficiently well behaved projected least squares estimator for $\tau^0,$ that is optimal ($T$-consistent) in its rate of convergence in the assumed high dimensional setting, while allowing the change point to potentially diminish under the restriction $\xi^{-1}\big\{s\log (p\vee T)\big/\surd T\}\le c,$ i.e., the jump size potentially being near the boundary of the ROD upto a factor of $\surd\big\{s\log (p\vee T)\big\}.$ Next, another important objective is to derive its limiting distribution in order to enable construction of confidence intervals for the change parameter $\tau^0.$ Other more subtle advantages of the methodology to be proposed are: (i) the ability to consistently filter out the case of $\tau^0=1,$ in a preliminary regularized estimation step, thus eliminating the need for pre-testing for the existence of a change point. This boundary case shall be excluded for the discussion in Section \ref{sec:intro}  and Section \ref{sec:mainresults} and shall be brought up in Section \ref{sec:nuisance}; (ii) Relaxing the assumption of gaussianity to subgaussianity, and additionally allowing for a general positive definite spatial dependence structure; Finally, (iii) to provide a computationally efficient and highly scalable methodology, specifically, the method to be proposed has no requirement of any algorithmic optimization for the entire procedure. Instead, we shall require only arithmetic operations and explicit identification of a minima amongst $T$ numbers for implementation of the proposed methods.

We begin with the necessary groundwork to proceed further. For any $z_t\in\R,$ $t=1,...,T,$ let $z=(z_1,...,z_T)^T,$ and for any $\theta_1,\theta_2\in\R,$ and $\tau\in(0,1),$ define the least squares loss,
\benr\label{eq:Q}
Q(z,\tau,\theta_1,\theta_2)=\frac{1}{T}\sum_{t=1}^{\lfloor T\tau\rfloor}(z_t-\theta_1)^2+\frac{1}{T}\sum_{t=\lfloor T\tau\rfloor+1}^{T}(z_t-\theta_2)^2.
\eenr
Additionally, let $\eta^0=\mu_1^0-\mu_2^0\in\R^p,$ $\theta_1^0=\eta^{0T}\mu_1^0\in\R,$ and $\theta_2^0=\eta^{0T}\mu_2^0\in\R.$ Then define a latent one dimensional projection of $y_t$ of (\ref{model:subgseries}) as,
\benr\label{mod:projectedseries}
z_t=\eta^{0T}y_t=\begin{cases}\theta^0_1+\psi_t, &t=1,...,\lfloor T\tau^0\rfloor\\
	\theta^0_2+\psi_t,& t=\lfloor T\tau^0\rfloor+1,...,T,\end{cases}
\eenr
where $\psi_t=\eta^0\vep_t,$ $t=1,...,T.$ Clearly the series $\{z_t\}_{1}^T$ is unobservable, since the nuisance parameters $\eta^0,$ is unknown. It may be of interest to note that the model (\ref{mod:projectedseries}) is the same latent projection that lies at the heart of the methodology of \cite{wang2018high}, wherein the authors proceed to recovery of the change point by seeking an optimal projection via a singular value decomposition together with a CUSUM transformation. In contrast, we take a more simpler route via least squares.

Now, suppose estimates $\h\mu_1,$ $\h\mu_2,$ are available such that with probability at least $1-o(1),$ the following bounds are satisfied.
\benr\label{eq:optimalmeans}
\|\h\mu_1-\mu_1^0\|_2\le c_u\si_{\vep}\Big(\frac{s \log (p\vee T)}{Tl_T}\Big)^{\frac{1}{2}},\quad {\rm and }\quad \|\h\mu_2-\mu_2^0\|_2\le c_u\si_{\vep}\Big(\frac{s \log (p\vee T)}{Tl_T}\Big)^{\frac{1}{2}}
\eenr
where $s$ is a sparsity parameter defined in Condition A of Section \ref{sec:mainresults}, and $0<l_T<1/2$ is sequence separating the unknown change point from the boundaries of $(0,1),$ i.e., $(\tau^0)\vee (1-\tau^0) \ge l_T.$ The parameter $\si_{\vep}$ is the variance proxy of the $p$-dimensional subgaussian vector $\vep_t$ (Condition B). The availability of these mean estimates is assumed only for the time being (Section \ref{sec:intro} and Section \ref{sec:mainresults}), and for the purpose of a clear presentation of the main idea enabling inference on $\tau^0.$ In Section \ref{sec:nuisance} we provide two distinct approaches to obtain such estimates via regularization.

Let $\h\eta=\h\mu_1-\h\mu_2,$ $\h\theta_1=\h\eta^T\h\mu_1$ and $\h\theta_2=\h\eta^T\h\mu_2.$ Then define the observable one dimensional surrogate $\h z=(\h z_1,...,\h z_T)^T,$ of $z,$ where $\h z_t=\h\eta^T y_t,$ $t=1,...,T.$ Under this setup we propose the projected least squares estimate defined as,
\benr\label{est:optimal}
\tilde\tau=\argmin_{\tau\in(0,1)} Q(\h z,\tau,\h\theta_1,\h\theta_2)
\eenr
The two distinctions between the estimator (\ref{est:optimal}) and the least squares estimator of \cite{bai2010common} and \cite{bhattacharjee2019change} are that, first, we use regularized mean estimates $\h\mu_1$ and $\h\mu_2$ satisfying (\ref{eq:optimalmeans}) in the construction of the proposed $\tilde\tau,$ in comparison to ordinary empirical means as considered in (\cite{bai2010common}) and \cite{bhattacharjee2019change}. This distinction allows control certain empirical processes that show up as residual terms in the estimation of $\tau^0.$ Secondly, the proposed $\tilde\tau$ estimate is extracted from a one dimensional projected series, instead of being extracted directly from the observed $p$-dimensional series. These improvements provide sufficient regularity to the change point estimate $\tilde\tau,$ and we shall show that despite using irregular estimates $\h\mu_1,$ and $\h\mu_2$ that are not root-$T$ consistent, the estimate $\tilde \tau$ satisfies an optimal rate of convergence, $T\xi^2(\tilde\tau-\tau^0)= O_p(1),$ under mild conditions. Furthermore, we shall obtain its limiting distribution, given by,
\benr\label{eq:limitingdist}
T\xi^2\si^{-2}(\tilde\tau-\tau^0)\Rightarrow \argmin_{v}\big(|v|-2W(v)\big)
\eenr
where $\si^2=\lim_{T\to\iny}(\eta^{0T}\Sigma_{\vep}\eta^0\big)/\xi^2,$ $\Sigma_{\vep}={\rm cov}(\vep_t),$ and $W(\cdot)$ is a two-sided Brownian motion on $\R.$ It may be observed that the limiting distribution obtained here is the same as that of the least squares estimate of $\tau^0$ in a one dimensional time series, (\cite{bai1994}). The distribution of  $\argmin_{v}\big(|v|-2W(v)\big)$ is infact well studied in the literature and approximations of its cumulative distribution function and thus its quantiles are readily available, (\cite{yao1987approximating}). Our results shall allow the validity of this discussion in the high dimensional regime under mild technical conditions. The jump size as before may potentially be near the boundary of the ROD.

It is fairly unusual for irregular estimates of some parameters of a model that are slower than root-$T,$ to yield an optimal estimate of another parameter of the model, as achieved by the proposed $\tilde\tau$ estimate. However, precedents for it are available in the recent high dimensional inference literature for static regression models. To describe this connection, first consider the following motivating heuristical insight. Localizing the change point estimate obtained from the projected series $\{\h z_t\}_1^T,$ requires control on a noise term of the form $|\sum_t \h\eta^T\vep_t|\big/T,$ which can be bounded as follows,
\benr\label{eq:term1}
\frac{1}{T}|\sum_t \h\eta^T\vep_t|\le\frac{1}{T}|\sum_t\eta^{0T}\vep_t|+ \|\h\eta-\eta^0\|_2\sup_{\delta\in\cA; \|\delta\|_2=1}\frac{1}{T}|\sum_t \delta^T\vep_t|,
\eenr
where $\cA$ is a convex subset of $\R^p$ to which $(\h\eta-\eta^0)\big/\|(\h\eta-\eta^0)\|_2$ can be restricted to using regularization (discussed in Section \ref{sec:nuisance}). For illustration purposes, consider the simplified case where $\xi=O(1)$ and $l_T\ge c>0.$ Then, clearly the first term on the right hand side (rhs) of (\ref{eq:term1}) is $O_p(1\big/\surd{T}).$ From (\ref{eq:optimalmeans}) we have $\|\h\eta-\eta^0\|_2=O_p\big[\surd \{s\log (p\vee T)/T\}\big],$ and finally it can also be shown that the empirical process in second term of the rhs of (\ref{eq:term1}) can be restricted to $O_p\big[\surd \{s\log (p\vee T)/T\}\big].$ This yields, $|\sum_t \h\eta^T\vep_t|= O_p\big(1/\surd{T}\big)+O_p\big\{s\log (p\vee T)\big/T\big\}=O_p\big(1/\surd{T}\big),$ under the rate assumption $s\log p\big/\surd{T}\to 0.$ Notice here that the noise term considered in (\ref{eq:term1}) can be controlled at an optimal $1/\surd T$ rate, despite irregular estimates $\h\mu_1,$ $\h\mu_2$ that are slower than root-T. Note that, by nature of the estimators of \cite{bai2010common} and \cite{bhattacharjee2019change} where ordinary empirical means are used, the same control on the desired noise process may not be achievable. Thus their methodologies instead require a much larger jump size so as to dominate such noise terms. This forms one of the main reasons for the proposed estimate to achieve the optimal rate, without assuming a diverging jump size. This effect is conceptually identical to that obtained by the use of orthogonal moment functions in the context of inference on regression parameters, which in the recent past have been utilized for the construction of confidence regions for mean parameters in high dimensional regression models, e.g. \cite{belloni2011inference}, \cite{belloniinference}, \cite{van2014asymptotically},  \cite{belloni2017confidence}, and \cite{ning2017general} among others.

We conclude this section with a note on the computation of $\tilde\tau.$ Given the availability of mean estimates $\h\mu_1$ and $\h\mu_2,$ observe that the least squares loss function $Q(\h z, \cdot,\h\mu_1,\h\mu_2)$ is a step function in the interval $(0,1),$ with step changes occurring at the grid points $\{1/T,2/T,....(T-1)/T\}.$ This observation reduces computation of (\ref{est:optimal}) to a discrete optimization on a one dimensional grid of $(T-1)$ points, i.e., we can equivalently obtain $\tilde\tau$ as,
\benr\label{est:discrete}
\tilde\tau=\argmin_{\tau\in\{\frac{1}{T},\frac{2}{T},...\frac{T-1}{T}\}}Q(\h z,\tau,\h\theta_1,\h\theta_2).
\eenr
This optimization can be implemented simply by calculating $Q(\h z, \tau,\h\mu_1,\h\mu_2),$ for each $\tau\in\{\frac{1}{T},\frac{2}{T},...\frac{T-1}{T}\}$ and then explicitly locating the minimizing argument, i.e, implementation of (\ref{est:discrete}) involves only $T$ arithmetic operations.

The following sections provide a rigorous presentation of the above discussion as well as the thus far disregarded aspect of obtaining computationally efficient nuisance estimates satisfying (\ref{eq:optimalmeans}), which can additionally filter out the `no change' case consistently.

\vspace{1.5mm}
\noi{\it Notation}: Throughout the paper, $\R$ represents the real line. For any vector $\delta\in\R^p,$ $\|\delta\|_1,$ $\|\delta\|_2,$ $\|\delta\|_{\iny}$ represent the usual 1-norm, Euclidean norm, and sup-norm respectively. For any set of indices $U\subseteq\{1,2,...,p\},$ let $\delta_U=(\delta_j)_{j\in U}$ represent the subvector of $\delta$ containing the components corresponding to the indices in $U.$ Let $|U|$ and $U^c$ represent the cardinality and complement of $U.$ We denote by $a\wedge b=\min\{a,b\},$ and $a\vee b=\max\{a,b\},$ for any $a,b\in\R.$ The notation $\lfloor \cdotp \rfloor$ is the usual greatest integer function. We use a generic notation $c_u>0$ to represent universal constants that do not depend on $T$ or any other model parameter. In the following this constant $c_u$ may be different from one term to the next. All limits in this article are with respect to the sample size $T\to\iny.$ We use the notation $\Rightarrow$ to represent convergence in distribution.

\section{Main results}\label{sec:mainresults}
In this section we state our assumptions and main theoretical results regarding $T$-consistency and the limiting distribution (\ref{eq:limitingdist}) of the project least squares estimator.

\vspace{1.5mm}
{\it {{\noi{\bf Condition A (assumption on model parameters):}} (i) Let $S=S_1\cup S_2,$ where $S_1=\{j;\mu^0_{1j}\ne 0\}$ and $S_2=\{j;\mu^0_{2j}\ne 0\}.$ Then for some $s=s_T\ge 1,$ we assume that $|S|\le s.$  (ii) The model dimensions $s,p,T,$ satisfy the rate $s\log p\big/\surd T\to 0.$
		(iii) Assume a change point exists and is sufficiently separated from the boundaries of $(0,1),$ i.e., for some positive sequence $l_T>0,$ we have $(\tau^0)\wedge (1-\tau^0) \ge l_T.$ Additionally, the jump vector $\eta^0=\mu_1^0-\mu_2^0$ is such that the jump size $\xi=\|\eta^0\|_2,$ together with $l_T$ satisfies the following restriction,
		\benr
		\frac{\si_{\vep}}{\xi}\Big\{\frac{s\log (p\vee T)}{\surd{T l_T}}\Big\} \le c_{u},\nn
		\eenr
		for an appropriately chosen small enough constant $c_{u}>0.$
}}

The sparsity assumption of Condition A(i) is typically made on the jump vector $\eta^0,$ as done in \cite{wang2018high} and \cite{enikeeva2013high}. In contrast we make this assumption directly on the mean vectors $\mu_1^0$ and $\mu_2^0.$ These two variations of the sparsity assumption are equivalent, which can be seen as follows. Consider $y_t$ of model (\ref{model:subgseries}) such that the jump $\eta^0$ is $s$-sparse, i.e., there is a mean change in at most $s$ components. Then upon centering $y_t$ with columnwise empirical means, $y_t^*=y_t-\bar y,$ $t=1,...,T,$ with $\bar y=\sum_{t=1}^T y_t\big/T,$ the sparsity of $\eta^0$ is transferred onto the new mean vectors $\mu_1^*=Ey_t^*,$ $t\le \lfloor T\tau^0\rfloor,$ and $\mu_2^*=Ey_t^*,$ $t>\lfloor T\tau^0\rfloor,$ in the sense of Condition A(i). All results of this article can also be developed by directly assuming sparsity of the jump vector. However we use Condition A(i) solely to easy notational complexity in some of the proofs. In the rest of this article we assume that the series $y_t$ has been centered, thus allowing Condition A(i) to be applicable. Condition A(ii) restricts the rate of divergence of model dimensions, this assumption is consistent with the recent literature on inference for regression coefficients in high dimensional linear regression models, see, e.g. \cite{belloni2017confidence}, and \cite{ning2017general} among others. Condition A(iii) assumes existence of a change point within the sampling period and its sufficient separation from the boundaries of $(0,1).$ This assumption is made for the inference methodology of this section. However, we shall relax this condition in Section \ref{sec:nuisance} to include $\tau^0=1$ in the prestep estimation process and thus filter out this case consistently before the inference methodology is implemented. The remaining assumptions of Condition A(iii) puts us in the regime where the jump size is potentially close to the boundary of the ROD upto a factor of $\surd\{s\log (p\vee T)\}.$ This condition is only marginally stronger than (17) assumed in \cite{wang2018high} and plays a key role in yielding optimality of the proposed projected least squares estimator. No assumption on upper bounds for the jump size are made.

\vspace{1.5mm}
{\it {{\noi{\bf Condition B (assumption on the model distribution):}} The vectors $\vep_t=(\vep_{t1},...,\vep_{tp})^T,$ $t=1,..,T,$ are i.i.d subgaussian with mean vector zero, and variance proxy $\si_{\vep}^2\le c_u.$ Furthermore, the covariance matrix $\Sigma_{\vep}:=E\vep_t\vep_t^T$ has bounded eigenvalues, i.e., $0<\ka\le\rm{min eigen}(\Si_{\vep})<\rm{max eigen}(\Si_{\vep})\le\phi<\iny.$}}

Condition B is fairly standard in the high dimensional literature. This condition assumes temporal independence and a general positive definite covariance structure spatially.  It does not require any specific spatial dependence structure such as those in \cite{liu2019minimax} or the assumption of gaussianity as considered in \cite{wang2018high}. More specifically, this condition serves two purposes. Firstly, it allows the residual process in the estimation of $\tilde\tau$ to converge weakly to the distribution in (\ref{eq:limitingdist}). Secondly, under a suitable choice of parameters, it allows estimation of nuisance parameters at the rates of convergence presented in (\ref{eq:optimalmeans}) by one of several estimators. For the presentation of this section we are agnostic about the choice of the nuisance estimator and instead require the following condition.

\vspace{1.5mm}
{\it {{\noi{\bf Condition C (assumption nuisance parameter estimates):}} Let $\Delta_T\to 0$ be a fixed sequence. Then with probability $1-\Delta_T,$ the estimators $\h\mu_1$ and $\h\mu_2$ satisfy (\ref{eq:optimalmeans}). Additionally, with the same probability, the vectors $(\h\mu_1-\mu_1^0),$ $(\h\mu_2-\mu_2^0)\in\cA.$ Here $\cA$ is a convex subset of $\R^p$ defined as, ${\cA}=\big\{\delta\in\R^p;\,\,\|\delta_{S^c}\|_1\le {\it c_u} \|\delta_S\|_1\big\},$ with $S$ being the set of indices defined in Condition A(i).}}

A few notations are necessary to proceed further. For any $z\in\R^T,$ and $\tau,\theta_1,\theta_2\in\R$ define,
\benr\label{def:cU}
\cU(z,\tau,\theta_1,\theta_2)=Q(z,\tau,\theta_1,\theta_2)-Q(z,\tau^0,\theta_1,\theta_2),\nn
\eenr
where $\tau^{0}\in(0,1)$ is the unknown change point parameter and $Q$ is the least squares loss as defined in (\ref{eq:Q}). Also, for any non-negative sequences $u_T,$ and $v_T,$ with  $v_T\le u_T,$ define the collection,
\benr\label{def:setcG}
\cG(u_T,v_T)=\Big\{\tau\in (0,1);\,\,Tv_T\le \big|\lfloor T\tau \rfloor-\lfloor T\tau^0\rfloor\big|< Tu_T\Big\}
\eenr

We begin with a lemma that provides a uniform lower bound on the expression $\cU(\h z, \tau,\h\theta_1,\h\theta_2),$ over the collection $\cG(u_T,v_T).$ This lower bound forms the basis of the argument used to obtain $T$-consistency of the proposed estimator.

\begin{lem}\label{lem:mainlowerb} Suppose Conditions A, B and C hold. Let $u_T$ and $v_T$ be any non-negative sequences and let $\cG(u_T,v_T)$ be as defined in (\ref{def:setcG}). Then for any $0<\gamma<1,$ there exists a constant $c_{u1},$ such that the following uniform lower bound holds.
	\benr
	\inf_{\tau\in\cG(u_T,v_T)}\cU(\h z,\tau,\h\theta_1,\h\theta_2)\ge c_u\xi^4\Big\{v_T-\frac{c_{u1}\si_{\vep}}{\xi}\Big(\frac{u_T}{T}\Big)^{\frac{1}{2}}\Big\},\nn
	\eenr
	with probability at least $1-\gamma-\Delta_T-o(1).$
\end{lem}

Our first main result to follow establishes the $T$-consistency of the projected least squares estimator $\tilde\tau.$ While the detailed proof of this result is provided in Appendix A, here we provide a brief sketch of the main idea. Overall the proof proceeds by a contradiction argument, which proves that the estimate $\lfloor T\tilde\tau\rfloor$ cannot lie anywhere except an $O(\xi^{-2})$ neighborhood of $\lfloor T\tau^0\rfloor,$ in probability. More specifically, using Lemma \ref{lem:mainlowerb} recursively, we show that for any $v_T$ slower in rate than $O(T^{-1}\xi^{-2})$ we have,
\benr
\inf_{\tau\in\cG(1,v_T)}\cU(\h z,\tau,\h\theta_1,\h\theta_2)>0,\nn
\eenr
in probability. Upon noting that by definition $\tilde\tau$ must satisfy $\cU(\h z,\tau,\h\theta_1,\h\theta_2)\le 0,$ the argument shall yield the desired $T$-consistency.

\begin{thm}\label{thm:optimalapprox} Suppose Conditions A, B and C hold. Then the projected least squares estimate $\tilde\tau$ satisfies the bound, $\si_{\vep}^{-2}\xi^2\big(\lfloor T\tilde\tau\rfloor-\lfloor T\tau^0\rfloor\big)=O_p(1).$
\end{thm}

A direct application of Theorem \ref{thm:optimalapprox} under a diverging jump size directly yields perfect identifiability of the change point on the integer valued scale. This is stated in the following corollary.

\begin{cor}\label{cor:perfectid} Suppose Conditions A, B and C hold and assume that $\xi\to \iny.$ Then,
	\benr
	pr\big(\lfloor T\tilde\tau\rfloor=\lfloor T\tau^0\rfloor\big)\to 1.\nn
	\eenr
\end{cor}

\begin{rem}\label{rem:optapprox} Note that the results of Theorem \ref{thm:optimalapprox} and Corollary \ref{cor:perfectid} are very similar to many results in the classical change point literature. However these result points towards the following subtlety regarding the rates of convergence of change point estimates in the integer $(\lfloor T\tilde\tau\rfloor)$ and continuous scales $(\tilde\tau)$ that has often been disregarded in the literature. Note that we have the deterministic inequality $\big(T(\tilde\tau-\tau^0)-1\big)\le \big(\lfloor T\tilde\tau\rfloor-\lfloor T\tau^0\rfloor\big)\le \big(T(\tilde\tau-\tau^0)+1\big).$ In the case where $\xi=O(1),$ an application of this inequality together with the result of Theorem \ref{thm:optimalapprox} directly implies that $T\xi^2(\tilde\tau-\tau^0)=O_p(1).$ However, when $\xi\to \iny,$ this may not be true. Instead, in this case we obtain $T(\tilde\tau-\tau^0)=O_p(1).$ Consequently, when $\xi\to \iny,$ while perfect identification (in probability) of the integer scale change point can be guaranteed using Theorem \ref{thm:optimalapprox}, the same cannot be said for the change point in the continuous scale, where the result of Theorem \ref{thm:optimalapprox} can only guarantee the rate $T(1\vee \xi^2)(\tilde\tau-\tau^0)=O_p(1).$
\end{rem}

Theorem \ref{thm:optimalapprox} establishes the optimality of the proposed method, despite irregular estimates $\h\mu_1,$ $\h\mu_2$ being used in its construction. Several important observations are discussed in the following. First and foremost is to note that Theorem \ref{thm:optimalapprox} is established while allowing the jump size to possibly be nearly at the boundary of the ROD \big(Condition A(iii)\big). An informative comparison illustrating the usefulness of the proposed estimator is with the least squares estimate applied to the entire $p$-dimensional data set, where \cite{bai2010common} requires a diverging jump size satisfying $\xi\big/\surd p\to \iny,$ and \cite{bhattacharjee2019change} require $\xi \surd(T/p)\to \iny,$ in order to obtain a similar optimality result. A closer comparison is with the related estimator of \cite{wang2018high}, which has been shown to satisfy the rate  $\log\log T/T$ (Theorem 1 of \cite{wang2018high}). In comparison to the optimal rate obtained by our estimator $\tilde\tau.$ However, we achieve this at the cost of being marginally further away from the boundary of the ROD by a factor of $\surd\{s\log (p\vee T)\}$ in comparison to their article. We believe that this points towards a delicate relationship between optimality of an estimator and the rate assumption on the jump size. Thus, it may very well be the case that the estimator of \cite{wang2018high} is also optimal under the marginally stronger Condition A(iii) on the jump size, as made in this article, especially since both methodologies are built upon the same latent projection (\ref{mod:projectedseries}). However, this connection is not pursued further in this article.

It may also be worth noting that the mean estimates $\h\mu_1$ and $\h\mu_2$ used to construct $\tilde\tau$ do not require to satisfy oracle type properties in the sense of \cite{fan2001variable}, i.e. the result of Theorem \ref{thm:optimalapprox} remains valid despite a potentially imperfect recovery of the support of $\mu_1^0$ and $\mu_2^0.$ Furthermore no minimum separation from zero conditions on the means $\mu_1^0$ and $\mu_2^0$ are assumed. This is again in coherence with recent developments for inference on regression coefficients in high dimensional linear regression models, see e.g. \cite{belloni2011inference}, \cite{belloni2017confidence}, \cite{van2014asymptotically} and \cite{ning2017general}.

Corollary \ref{cor:perfectid} provides the degenerate limiting behavior of the change point estimate on the integer valued scale. While the final conclusion of the result is identical to Theorem 3.2 of \cite{bai2010common} and Theorem 2.2(a) of \cite{bhattacharjee2019change},  the important distinguishing factors from those articles is again its applicability under (i) much weaker jump signal and (ii) high dimensional setting, respectively. The following result establishes the limiting distribution of the proposed estimate $\tilde\tau,$ in the regime where the jump size diminishes with $T.$

\begin{thm}\label{thm:limitingdist} Suppose Conditions A, B and C hold, and assume that,
	\benr \label{asm:jumpnew}
	\Big(\frac{\si_{\vep}}{\xi}\Big)\Big(\frac{s\log (p\vee T)}{\surd{Tl_T}}\Big) =o(1).
	\eenr
	Additionally assume that the jump size $\xi\to 0,$ and that  $(\eta^{0T}\Sigma_{\vep}\eta^0\big)\big/\xi^2\to \si^2.$ Then the projected least squares estimate $\tilde\tau$ obeys the following limiting distribution.
	\benr
	T\xi^2\si^{-2}(\tilde\tau-\tau^0)\Rightarrow \argmin_{r}\big(|r|-2W(r)\big),\nn
	\eenr
	where $W(\cdot)$ is a two sided Brownian motion\footnote{A two-sided Brownian motion $W(\cdot)$ is defined as $W(0)=0,$ $W(r)=W_1(r),$ $r>0$ and $W(r)=W_2(-r),$ $r<0,$ where $W_1(\cdot)$ and $W_2(\cdot)$ are two independent Brownian motions defined on the non-negative half real line, see e.g. \cite{bai1994} and \cite{bai2010common}.} on $\R.$
\end{thm}
The assumption (\ref{asm:jumpnew}) is slightly stronger than that assumed in Condition A(iii) and is required to obtain the given limiting distribution. This requirement of a marginally stronger assumption in comparison to that required for $T$-consistency is inline with the classical literature, see, e.g. Condition (C) of \cite{bai1994}. The only condition here that may be considered out of the ordinary is $(\eta^{0T}\Sigma_{\vep}\eta^0\big)\big/\xi^2\to \si^2.$ However this is a mild requirement given that under Condition B we have  $\ka^2\xi^2\le(\eta^{0T}\Sigma_{\vep}\eta^0\big)\le \phi^2\xi^2.$ Note also that $(\eta^{0T}\Sigma_{\vep}\eta^0\big)$ is the expression for the variance of $\psi_t$ in the latent model (\ref{mod:projectedseries}).

The limiting distribution presented in Theorem \ref{thm:limitingdist} is classical and has been studied extensively in the literature, see, e.g. \cite{picard1985testing}, \cite{yao1987approximating}, \cite{bai1994}, \cite{bai2010common}, \cite{jandhyala2013inference} among several others. The form of the distribution function is available in \cite{yao1987approximating}. This enables straightforward computation of quantiles, which can in turn be utilized to construct confidence intervals with any desired asymptotic coverage level.

\section{Computationally efficient nuisance parameter estimation via regularization}\label{sec:nuisance}

The main objective of this section is to provide computationally efficient estimates $\h\mu_1,$ $\h\mu_2$ for the nuisance parameters $\mu_1^0$ and $\mu_2^0,$ that satisfy the requirements of Condition C, so that the proposed projected least squares estimator of the previous section is viable. We propose using regularized stopped time estimates, in particular, soft-thresholded empirical means computed on each binary partition yielded by a preliminary near optimal change point estimate. This preliminary change point estimate can be an existing one from the literature, such as that of \cite{wang2018high}, or the new near optimal estimator proposed later in this section. A comparison between these two distinct approaches is also provided later in this section.

We begin by defining soft-thresholded empirical means. For this purpose we require the following notation. For any $\tau\in(0,1),$ such that $\lfloor T\tau\rfloor\ge 1,$ let $\bar y_{(0:\tau]},$ and $\bar y_{(\tau:1]}$ represent the stopped time empirical means defined as,
\benr\label{def:empmeans}
\bar y_{(0:\tau]}=\frac{1}{\lfloor T\tau \rfloor}\sum_{t=1}^{\lfloor T\tau\rfloor} y_t,\quad{\rm and}\quad \bar y_{(\tau:1]}=\frac{1}{T-\lfloor T\tau\rfloor}\sum_{t=\lfloor T\tau\rfloor+1}^{T} y_t.
\eenr
Next consider the soft-thresholding operator, $k_{\la}(x)={\rm sign}(x)(|x|-\la)_{+},$ $\la>0,$ $x\in\R^p,$ where ${\rm sign}(\cdotp)$ and $|\cdotp|$ are applied component-wise. Then for any $\la_1,\la_2>0,$ define regularized mean estimates,
\benr\label{est:softthresh}
\h\mu_1(\tau)=k_{\la_1}\big(\bar y_{(0:\tau]}\big),\quad{\rm and}\quad \h\mu_2(\tau)=k_{\la_2}\big(\bar y_{(\tau:1]}\big),
\eenr
Clearly, these soft thresholded estimates are negligible in their computational complexity, all they require are two arithmetic operations, namely, computation of the empirical mean and the operator $k_{\la}(\cdot).$ In practice, the only significant computation time required here would be that of cross validation or other tuning parameter selection processes.
It is well known in the literature (\cite{donoho1995noising}, \cite{donoho1995wavelet}) that the soft-thresholding operation in (\ref{est:softthresh}) is equivalent to the following $\ell_1$ regularization.
\benr\label{est:softL1construction}
\h\mu_1(\tau)&=&\argmin_{\mu_1\in\R^p}\big\|\bar y_{(0:\tau]}-\mu_1\big\|^2_2+\la_1\|\mu_1\|_1,
\eenr
and similarly for $\h\mu_2(\tau).$ The following result provides a uniform bound on these soft-thresholded means that reduce the problem of obtaining $\h\mu_1$ and $\h\mu_2$ that satisfy Condition C, to obtaining a preliminary near optimal estimate of the change point.

\begin{thm}\label{thm:unifmean} Suppose Condition B holds and let $\tau^0\in(0,1].$ Then we have the following.  \\
	(i) When $\tau^0=1,$ let $\mu_2^0=\mu_1^0$ and $\la_1=\la_2=c_u\si_{\vep} \surd\big\{s\log (p\vee T)\big/ Tl_T\big\}.$ Then for all $\tau\in(0,1)$ with $\tau\wedge(1-\tau)\ge c_ul_T,$ we have $\big\{\h\mu_1(\tau)-\mu_1^0\big\}, \big\{\h\mu_2(\tau)-\mu_2^0\big\}\in\cA,$ and for $q=1,2,$
	\benr
	\sup_{\substack{\tau\in(0,1)\\ \tau\wedge(1-\tau)\ge c_ul_T}} \|\h\mu_1(\tau)-\mu_1^0\|_q\le c_u\si_{\vep}s^{\frac{1}{q}}\Big\{\frac{\log (p\vee T)}{Tl_T}\Big\}^{\frac{1}{2}},\nn
	\eenr
	with probability at least $1-c_{u1}\exp\{-c_{u2}\log (p\vee T)\}.$ \\
	(ii) When $\tau^0<1,$ let $u_T>0$ be any sequence, and $\xi_{\iny}>0$ be such that $\|\eta^0\|_{\iny}\le \xi_{\iny}.$ Additionally let  	
	\benr
	\la_1=\la_2=c_u\max\Big[\si_{\vep}\Big\{\frac{\log (p\vee T)}{Tl_T}\Big\}^{\frac{1}{2}},\,\frac{\xi_{\iny} u_T}{l_T}\Big].\nn
	\eenr
	Then for all $\tau\in \cG(u_T,0),$ with $\tau\wedge(1-\tau)\ge c_ul_T,$ we have $\big\{\h\mu_1(\tau)-\mu_1^0\big\}, \big\{\h\mu_2(\tau)-\mu_2^0\big\}\in\cA,$ and for $q=1,2,$ 	\benr
	\sup_{\substack{\tau\in\cG(u_T,0)\\ \tau\wedge(1-\tau)\ge c_ul_T}} \|\h\mu_1(\tau)-\mu_1^0\|_q\le c_us^{\frac{1}{q}}\max\Big[\si_{\vep}\Big\{\frac{\log (p\vee T)}{Tl_T}\Big\}^{\frac{1}{2}},\,\frac{\xi_{\iny} u_T}{l_T}\Big],\nn
	\eenr	
	with probability at least $1-c_{u1}\exp\{-c_{u2}\log (p\vee T)\}.$ The same uniform upper bounds also hold for $\|\h\mu_2(\tau)-\mu_2^0\|_q,$ $q=1,2.$
\end{thm}
The result of Theorem \ref{thm:unifmean} provides validity of these regularized stopped time mean estimates. Its usefulness is quite apparent. Consider any preliminary near optimal estimator $\h\tau,$ that satisfies,
\benr\label{est:nearoptratecp}
\big|\lfloor T\h\tau\rfloor-\lfloor T\tau^0\rfloor\big|\le c_u\si_{\vep}^2\xi^{-2} s\log (p\vee T),
\eenr
with probability at least $1-o(1),$ i.e., $\h\tau\in \cG(u_T,0),$ with  $u_T=c_u\si_{\vep}^2\xi^{-2} s\log (p\vee T)\big/T,$ with the same probability. Then, under the assumption $\xi^{-1}\surd(s\log p)\big/\surd (Tl_T)= o(1),$ the result of Theorem \ref{thm:unifmean} ensures that the mean estimates $\h\mu_1(\h\tau),$ and $\h\mu_2(\h\tau)$ satisfy all requirements of Condition C. Consequently, these requirements are now reduced to the availability of a preliminary near optimal estimator of the change point satisfying (\ref{est:nearoptratecp}). One example of such an estimator of the change point is that obtained from Algorithm 3 of \cite{wang2018high}. This is stated precisely in the following corollary along with the required assumptions from their article.

\begin{cor}\label{cor:meanws} Suppose the model (\ref{model:subgseries}), and assume $\vep_t\sim^{i.i.d.} \cN(0,\si_{\vep}^2I_{p\times p}),$ $t=1,...,T.$ Let $\tau^0\wedge(1-\tau^0)\ge l_T,$ and $\|\mu_1^0-\mu_2^0\|_2\ge \xi,$ and assume that,
	\benr\label{asm:wsasm}
	\frac{\si_{\vep}}{\xi l_T}\Big\{\frac{s\log \big(p \log T\big)}{T}\Big\}^{\frac{1}{2}}= o(1).
	\eenr
	Let $\h\tau$ be the estimate obtained from Algorithm 3 of \cite{wang2018high}. Then upon choosing $\la_1$ and $\la_2$ as prescribed in Theorem \ref{thm:unifmean}, with $u_T=\si^{2}\xi^{-2} \log (\log T)\big/ T,$ the soft thresholded estimates $\h\mu_1=\h\mu_1(\h\tau),$ and $\h\mu_2(\h\tau)$ satisfy the requirements of Condition C. \footnote{There are a few additional minor requirements for the validity of $\h\tau,$ of \cite{wang2018high} which appear to be artifacts of their proof, we refer to Theorem 1 of \cite{wang2018high} for further details.}.
\end{cor}

\begin{rem} It may be observed that the assumption (\ref{asm:wsasm}) is slightly stronger than the assumption $\si_{\vep}(\xi l_T)^{-1} \surd\big[\{s\log \big(p \log T\big)\}\big/ T ]\le c_u,$ of \cite{wang2018high}. While the latter is sufficient to obtain a near optimal rate of convergence of the change point estimate $\h\tau,$ we require the marginally stronger version (\ref{asm:wsasm}) to allow near optimality of the change point estimate to extend to the mean estimates $\h\mu_1$ and $\h\mu_2.$ Here we also mention that article of \cite{wang2018high} also extends their result to some spatial and temporal dependence structures.
\end{rem}

The availability of estimators $\h\mu_1$ and $\h\mu_2$ of Corollary \ref{cor:meanws} satisfying the requirements of Condition C makes the inference methodology of Section \ref{sec:mainresults} viable, and allows us to achieve the larger objective of this article, i.e., to perform inference on the location of the change point $\tau^{0}$ of the high dimensional model (\ref{model:subgseries}), while allowing the change point to be potentially near the boundary of the ROD. However, there still remains a significant room for improvement in several aspects of the nuisance parameter estimation methodology. For this purpose we propose an alternative near optimal estimator for the change point parameter and consequently alternative mean estimates $\h\mu_1$ and $\h\mu_2.$ This is discussed in the following.

In the remainder of this article we provide an alternative near optimal nuisance estimation methodology. The method to follow provides the following advantages in comparison to the existing literature. (a) It is applicable under a subgaussian assumption and allows for any general positive definite spatial dependence structure. (b) It is applicable even if $\tau^0=1,$ and infact provides consistent detection of this `no change' case, thus eliminating the need to pretest for existence of a change. Finally, (c) it is highly scalable and thus applicable to very large data sets. The cost associated with gaining these advantages shall only be a marginally stronger restriction on the sparsity parameter $s,$ and the separation from boundary sequence $l_T.$

A few more notations are necessary to describe our approach. Define the $p$-dimensional generalization of the least squares loss $Q$ of (\ref{eq:Q}), i.e., for any  $y_t\in\R^p,$ let $y=(y_1,...,y_T)^T\in\R^{T\times p},$ and for any vectors $\mu_1,\mu_2,\in\R^p,$ and any $\tau\in(0,1]$ define,
\benr\label{def:Qpdim}
Q(y,\tau,\mu_1,\mu_2)=\frac{1}{T}\sum_{t=1}^{\lfloor T\tau\rfloor}\big\|y_t-\mu_1\big\|_2^2+\frac{1}{T}\sum_{t=\lfloor T\tau\rfloor+1}^{T}\big\|y_t-\mu_2\big\|_2^2,
\eenr
where the second term in the rhs of (\ref{def:Qpdim}) is defined to be zero at $\tau=1.$ Also define a modified $\ell_0$-norm on $(0,1],$ as $\|\tau\|_0^*=0,$ if $\tau=1$ and $\|\tau\|_0^*=1,$ if $\tau<1.$ Then we propose Algorithm 1 to obtain a computationally efficient near optimal estimate of the change point parameter.

\vspace{-2mm}
\begin{figure}[H]
	\noi\rule{\textwidth}{0.5pt}
	
	\vspace{-1mm}
	\noi {\bf Algorithm 1:} Detection and near optimal estimation of the change point parameter
	
	\vspace{-3mm}
	\noi\rule{\textwidth}{0.5pt}
	
	\noi{\bf Step 0 (Initialize):} Choose any value $\check\tau\in(0,1),$ satisfying Condition D below, and compute mean estimates $\check\mu_1=\h\mu_1(\check\tau),$ and $\check\mu_2=\h\mu_2(\check\tau)$ using soft-thresholding, as defined in (\ref{est:softthresh}).
	
	\noi{\bf Step 1:} Update $\check\tau$ to obtain the change point estimate $\h\tau$ where,
	\benr
	\h\tau=\argmin_{\tau\in (0,1]}\Big\{Q(y,\tau,\check\mu_1,\check\mu_2)+\gamma \|\tau\|_0^*\Big\},\quad \gamma>0.\hspace{0.75in}\nn
	\eenr
	
	\vspace{-2mm}
	\noi\rule{\textwidth}{0.5pt}
\end{figure}
To complete the description of Algorithm 1, we first provide Condition D, which is a mild initializing condition of Step 0, and is satisfied by nearly any arbitrarily chosen $\check\tau\in(0,1),$ that is marginally away from the boundaries of this set.

\vspace{1.5mm}
{\noi {\bf Condition D}}:  Let $\check u_T$ be any non-negative sequence defined as,
\benr
\check u_T=1\wedge c_u\Big(\frac{1}{T}\Big)^{\frac{1}{k}},\quad {\rm for\,\, any\,\, constants,}\,\,k\in[1,\iny),\,\,{\rm and}\,\,c_u>0.\nn
\eenr
Then assume that the initializer $\check\tau$ satisfies,
\benr
\tau \vee (1-\tau)\ge c_ul_T,\quad{\rm and},\quad |\tau-\tau^0|\le \check u_T,\nn
\eenr
where $l_T$ is any sequence satisfying the rate assumptions of Condition A1.

A detailed discussion illustrating the mildness of this condition has been provided in Appendix D of the supplementary materials. Additionally, a brief summary of Condition D is provided right after the following inter-related condition that is required for the theoretical validity of Algorithm 1. The condition to follow is a weaker version of Condition A of Section \ref{sec:mainresults} in terms of $\xi$ and $p,$ and requires marginally stronger restrictions on the sparsity parameter $s$ and the sequence $l_T.$

\vspace{1.5mm}
{\it {\noi\bf Condition A1:} Suppose condition A(i), additionally assume the following. Let $\tau^0\in(0,1],$ and assume that if a change point exists, i.e., when $\tau^0<1,$ then $(\tau^0)\wedge (1-\tau^0)\ge l_T,$ for the same $l_T$ as of Condition D\footnote{For notational simplicity we assume $l_T$ to be the same sequence in both Condition A1 and D. This can be instead relaxed to only assuming the same order of these sequences.}. Additionally let $\|\eta^0\|_2\ge \xi,$ and $\|\eta^0\|_{\iny}\le \xi_{\iny}$ for any positive sequences $\xi$ and $\xi_{\iny}.$ Furthermore, assume that these sequences satisfy the following rate conditions,
	\benr
	(i)\,\,\frac{\si_{\vep}}{\xi}\Big\{\frac{s\log (p\vee T)}{Tl_T}\Big\}^{\frac{1}{2}}=o(1),\quad{and}\quad (ii)\,\,\frac{\surd{s}\xi_{\iny}}{\xi l_T T^{\frac{1}{k}}}\le c_{u}\nn
	\eenr
	for appropriately chosen small enough constant $c_{u}>0,$ where $k$ is the constant specified in Condition D.}

We begin by emphasizing the mildness of the initializing Condition D and that nearly any user chosen $\check\tau$ will satisfy it. The first part of this requirement only assumes that the initial choice $\check\tau$ is marginally away from the boundaries of $(0,1),$ and is clearly innocuous. For the second part of this condition, the key is to note that the constant $k$ may be arbitrarily large subject to $k$ satisfying the rate restriction in Condition A1. The usefulness of this flexibility is that $k$ can itself depend on the initial user chosen $\check\tau,$ i.e., the farther the initial guess $\check \tau,$ the larger $k$ can be in order to satisfy Condition D. We shall show in the following that the rate of convergence of the estimate $\h\tau$ obtained from Step 1 of Algorithm 1, shall be free of $k.$ This implies that the rate of convergence of $\h\tau$ does not depend on the precision of the user chosen initializer. Following is a simplified example that clearly illustrates the mildness of Condition D. First note that the restriction (ii) of Condition A1 can be simplified to $\surd{s}\big/\big(l_T T^{1/k}\big)\le c_{u}.$ Consider the case where $s\le c_u\log T,$ and $l_T\ge 1\big/c_u\log T.$ Now choose any $0<c_1<0.5,$ then any $\check\tau\in(c_1,1-c_1)$ will satisfy Condition D for some large enough constant $k>0,$ furthermore any such constant $k$ will in turn satisfy the rate condition of Condition A1 for $T$ sufficiently large, and thus will be a theoretically valid choice for the initializer of Algorithm 1.

Simply stated, this roughly implies that Algorithm 1 initialized with any user chosen $\check\tau\in(0,1)$ shall yield an estimate $\h\tau$ that lies in a near optimal neighborhood of $\tau^0.$ The restriction (ii) of Condition A1 also brings out the following closely related subtle observation. Suppose $\surd{s}\xi_{\iny}\le c_u\xi,$ then (ii) of Condition A1 becomes free of the sparsity parameter $s,$ consequently allowing the user chosen $\check \tau$ to be no longer restricted by the sparsity $s.$ This points to an interesting observation that the proposed Algorithm 1 can allow a larger number of changes when these jumps are evenly spread out across $s$ components of the jump vector $\eta,$ as opposed to unevenly large jumps in a few of these $s$ components. Nevertheless, Condition A1 allows the jump size to reach to the boundary of the ROD, upto the separation sequence $l_T$ and logarithmic terms in $s$ and $T.$ Additionally, this condition allows for the `no change' case, i.e., $\tau^0=1,$ which was absent from Condition A. We can now state the following result which provides the theoretical validity of the estimate $\h\tau$ of Algorithm 1.

\begin{thm}\label{thm:alg1cp} Suppose Conditions A1, B and D hold and choose $\la_1,\la_2$ as prescribed in Theorem \ref{thm:unifmean} with $u_T=\check u_T$ for Step 0, and $\gamma=c_u\si_{\vep}\xi \surd \{s\log (p\vee T)\big/T\}$ for Step 1. Then the estimate $\h\tau$ of Algorithm 1 satisfies the following relations.\\
	(i) When $\tau^0=1,$ then $\h\tau=1,$ with probability at least $1-c_{u1}\exp\{-c_{u2}\log (p\vee T)\}.$\\
	(ii) When $\tau^0<1,$ then,
	\benr
	\big|\lfloor T\h\tau\rfloor-\lfloor T\tau^0\rfloor\big|\le c_u\si_{\vep}^2\xi^{-2}s\log (p\vee T),\nn
	\eenr
	with probability at least $1-c_{u1}\exp\{-c_{u2}\log (p\vee T)\}.$
\end{thm}
This result provides the applicability of Algorithm 1, without prior knowledge on the existence of a change. In the case of `no change', $\h\tau$ of Algorithm 1 provides consistent detection of this case. In the case where a change exists, it yields an estimate that lies in a near optimal neighborhood of the unknown change point. Importantly, the selection consistency and the rate of convergence of $\h\tau,$ are free of the constant $k$ of Condition D. Perhaps surprisingly, this implies that the estimate $\h\tau$ of Step 1 of Algorithm 1 is not influenced by the precision of the initial guess $\check \tau.$ Roughly speaking, this result is counterintuitive in the sense that it says a `bad' initial guess in Step 0, will yield an estimate that is no worse in its rate of convergence than that obtained by starting the algorithm even at the true value $\tau^0.$ An illustration of this surprising result is provided in Figure \ref{fig:insensitivity1} in Appendix D of the supplementary materials..

The usefulness of Theorem \ref{thm:alg1cp} in context of the inference problem of Section \ref{sec:mainresults} are the following. (a) If $\tau^0=1,$ then we will consistently recover $\h\tau=1,$ consequently, there is no further need to proceed to the inference methodology of Section \ref{sec:mainresults}. (b) If  $\tau^0<1,$ then $\h\tau$ provides a preliminary near optimal estimate $\h\tau,$ which can in turn be utilized to obtain the desired nuisance estimates $\h\mu_1=\h\mu_1(\h\tau),$ and $\h\mu_2=\h\mu_2(\h\tau)$ satisfying Condition C, thus making the methodology of Section \ref{sec:mainresults} viable. More specifically, for case (b), we have the following corollary which is a direct application of Theorem \ref{thm:unifmean} and Theorem \ref{thm:alg1cp}.

\begin{cor}\label{cor:final} Suppose the conditions of Theorem \ref{thm:alg1cp} and assume that a change point exists, i.e., $\tau^0<1.$ Let $\h\tau$ be the estimate obtained from Algorithm 1 and $\h\mu_1=\h\mu_1(\h\tau),$ and $\h\mu_2=\h\mu_2(\h\tau),$ be the corresponding regularized stopped time mean estimates. Then upon choosing $\la_1,$ and $\la_2$ as prescribed in Theorem \ref{thm:unifmean} with $u_T= c_u\si_{\vep}^2\xi^{-2}s\log (p\vee T)\big/T$ we have that $\h\mu_1,$ $\h\mu_2\in\cA,$ for $\cA$ as defined in Condition C. Additionally upon assuming,
	\benr\label{asm:additional}
	\frac{s\si_{\vep}\xi_{\iny}}{\xi^2}\Big\{\frac{\log (p\vee T)}{Tl_T}\Big\}^{\frac{1}{2}}\le c_u,
	\eenr
	the following bounds hold for $q=1,2,$
	\benr
	\|\h\mu_1-\mu_1^0\|_q\le c_u\si_{\vep}s^{\frac{1}{q}}\Big\{\frac{\log (p\vee T)}{Tl_T}\Big\}^{\frac{1}{2}},\,\,\,{and}\,\,\,\|\h\mu_2-\mu_2^0\|_q\le c_u\si_{\vep}s^{\frac{1}{q}}\Big\{\frac{\log (p\vee T)}{Tl_T}\Big\}^{\frac{1}{2}}\nn
	\eenr
	with probability at least $1-c_{u1}\exp\{-c_{u2}\log (p\vee T)\}.$
\end{cor}

The above results provide all necessary machinery required to detect, estimate and perform inference on the change point parameter of the assumed model (\ref{model:subgseries}). We conclude this section with a final note on the computational efficiency of the proposed methodology. As discussed earlier, for each fixed $\la_1$ and $\la_2,$ Step 0 of Algorithm 1 is simply four arithmetic operations, same holds true for obtaining $\h\mu_1,$ and $\h\mu_2$ of Corollary \ref{cor:final}. Step 1 of Algorithm 1, and the projected least squares optimization in (\ref{est:optimal}) can be reduced to explicit identification of minima amongst $T+1$ numbers, as described earlier in (\ref{est:discrete}). This simplicity of computation allows the proposed methods to be implemented on potentially very large data sets.

\section{Numerical Results}\label{sec:numerical}

This section empirically illustrates the results developed in the preceding sections. The three main objectives of this section are the following: (i) to evaluate the estimation performance of the proposed projected least squares ($PLS$) estimator, and the new nuisance parameter estimation methodology (Algorithm 1, referred to as $AL1$ in the following), while benchmarking the performance of these methods with the estimator ($WS$) of \cite{wang2018high}. (ii) To evaluate the detection performance method $AL1$, i.e., to evaluate its ability to consistently detect the existence of a change point. Finally, (iii) to evaluate the empirical inference performance of the proposed $PLS$ estimator. The $PLS$ method is applied in conjunction with $AL1,$ which is used to obtain nuisance estimates $\h\mu_1,\h\mu_2$ required for the implementation of $PLS,$ in keeping with the result of Theorem \ref{cor:final}. In all simulations we assume no prior knowledge of any underlying parameters, in particular the method $AL1$ is initialized with $\check\tau=0.5$ irrespective of the value of $\tau^0.$ The first two objectives listed above are provided in Simulation A, and the results for the latter objective are provided in Simulation B.

In all our simulation designs, the unobserved noise variables $\vep_t$ are generated as independent Gaussian r.v.'s, more precisely we set $\vep_t\sim \cN(0,\Si),$ where $\Si$ is a $p\times p$ matrix with elements $\Si_{ij}=\rho^{|i-j|},$ and $\rho=0.5.$ The mean parameters of the model are set to be $\mu_1^0=(1_{1\times s},0_{p-s})^T_{p\times 1}$ and $\mu_2=(0_{1\times s},1_{1\times s},0_{p-2s})^T_{p\times 1},$ with $s=5.$ We let the model dimension to be $p\in\{50,500,750\}.$ The remaining specifications of Simulation A and Simulation B are as follows. For Simulation A we consider two cases, Simulation A.I considers $\tau^0\in\{0.2,0.4,0.6,0.8\}$ and evaluates the estimation performance of the $PLS,$ $AL1$ estimators, while benchmarking against the $WS$ estimator. Simulation A.II evaluates the detection ability of method $AL1$ by considering the cases $\tau^0\in\{0.8,1\},$ the first case meant to evaluate the true positive rate (existence of a change point is correctly identified) and the case of $\tau^0=1$ to evaluate the true negative rate. In both cases of Simulation A, we consider the sample size $T\in\{100,225,350\}.$ The tuning parameters $\la_1,$ $\la_2$ and $\gamma$ of the method $AL1$ are chosen adaptively using a BIC type criteria, the pertinent details of which are provided in Appendix D of the supplementary materials.

Simulation B is dedicated to evaluating the inference performance of the $PLS$ estimator. Using Theorem \ref{thm:limitingdist} we construct confidence intervals $\big[(T\tilde\tau-c_{\alpha}\si^2/\xi^2),\, (T\tilde\tau+c_{\alpha}\si^2/\xi^2)\big],$ for the change point parameter in the integer scale ($T\tau^0$), where $c_{\alpha}$ represents the $1-\alpha$ level critical value of the limiting distribution in (\ref{eq:limitingdist}). This critical value is evaluated as $c_{\alpha}=11.03$ using its distribution function provided in \cite{yao1987approximating}. For implementation of the confidence interval, we utilize plugin estimates of $\si^2$ and $\xi^2$ whose computational details are provided in Appendix D of the supplementary materials. In this simulation we consider $\tau^0\in\{0.2,0.4,0.6,0.8\}$ and the sample size  $T=350.$ In all cases of this simulation we construct $95\%$ confidence intervals ($\al=0.05$). For this simulation, we switch off the detection ability of $AL1$ by setting $\gamma=0.$ This is done solely to obtain as many confidence intervals as the number of replications (i.e., to avoid false negatives from $AL1$). In practice, when the $AL1$ methods detects $\h\tau=1,$ one would no longer pursue the inference objective of obtaining a confidence interval for its location.

To report our results we provide the following metrics which are computed based on 100 (for Simulation A) or 500 (for Simulation B) monte carlo repetitions: bias ($|E(\h\tau-\tau^0)|$), root mean squared error (RMSE, $\big\{E(\h\tau-\tau^0)^2\big\}^{1/2}$), true positive rate (TPR, $pr(\h\tau<1),$ when $\tau^0<1$) true negative rate (TNR, $pr(\h\tau=1),$ when $\tau^0=1$), coverage (relative frequency of the number of times $\tau^0$ lies in the confidence interval), and standard error (SE, average over replications of the computed standard error of $T\tilde\tau,$ i.e, $\h\si^2/\h\xi^2$).
\begin{table}
	 \resizebox{1\textwidth}{!}{
	\begin{tabular}{cccccccc}
			\hline
			\multicolumn{2}{c}{$\tau^0=0.2,$ $s=5$} & \multicolumn{2}{c}{$AL1$}                   & \multicolumn{2}{c}{$PLS$}                   & \multicolumn{2}{c}{$WS$}                    \\ \hline
			$T$                & $p$                & {\bf bias} ($\times 10^2$) & {\bf RMSE} ($\times 10^2$) & {\bf bias} ($\times 10^2$) & {\bf RMSE}  ($\times 10^2$) &{\bf bias} ($\times 10^2$)& {\bf RMSE}  ($\times 10^2$) \\ \hline
			100                & 50                 & 1.480                & 4.025                & 0.300                & 2.035                & 0.060                & 1.549                \\
			100                & 500                & 0.760                & 2.874                & 0.280                & 1.435                & 0.730                & 3.312                \\
			100                & 750                & 0.050                & 1.404                & 0.130                & 1.127                & 0.720                & 2.915                \\ \hline
			225                & 50                 & 0.556                & 1.172                & 0.018                & 0.377                & 0.031                & 0.578                \\
			225                & 500                & 0.307                & 0.748                & 0.116                & 0.431                & 0.084                & 0.385                \\
			225                & 750                & 0.440                & 1.977                & 0.062                & 0.507                & 0.049                & 0.442                \\ \hline
			350                & 50                 & 0.311                & 0.698                & 0.003                & 0.223                & 0.009                & 0.227                \\
			350                & 500                & 0.197                & 0.440                & 0.009                & 0.178                & 0.023                & 0.323                \\
			350                & 750                & 0.409                & 1.007                & 0.029                & 0.316                & 0.014                & 0.281                \\ \hline
	\end{tabular}}
	\label{tab:estimation}
\caption{\footnotesize{Results of Simulation A.I: estimation performance of $AL1,$ $PLS$ and $WS$ methods. Here, bias ($|E(\h\tau-\tau^0)|$), and root mean squared error (RMSE, $\big\{E(\h\tau-\tau^0)^2\big\}^{1/2}$).}}
\end{table}

\begin{table}
	\centering
	 \resizebox{0.57\textwidth}{!}{
		\begin{tabular}{ccccccc}
			\hline
			\multirow{2}{*}{$s=5$} & \multicolumn{3}{c}{$\tau^0=1$}             & \multicolumn{3}{c}{$\tau^0=0.8$}           \\ \cline{2-7}
			& $p=50$       & $p=500$      & $p=750$      & $p=50$       & $p=500$      & $p=750$      \\ \hline
			$T$                    & \textbf{TNR} & \textbf{TNR} & \textbf{TNR} & \textbf{TPR} & \textbf{TPR} & \textbf{TPR} \\ \hline
			100                    & 1            & 0.96         & 0.98         & 1            & 0.83         & 0.77         \\
			225                    & 1            & 1            & 1            & 1            & 1            & 1            \\
			350                    & 1            & 1            & 1            & 1            & 1            & 1            \\ \hline
	\end{tabular}}
	\label{tab:detection}
	\caption{\footnotesize{Results of Simulation A.II: evaluation of $AL1$ method for detection of a change point. Here, true positive rate (TPR, $pr(\h\tau<1),$ when $\tau^0<1$) true negative rate (TNR, $pr(\h\tau=1),$ when $\tau^0=1$).}}
\end{table}
\begin{table}
	\centering
		 \resizebox{0.8\textwidth}{!}{
		\begin{tabular}{ccccccc}
			\hline
			$T=350,$ $s=5$  & \multicolumn{2}{c}{$p=50$}      & \multicolumn{2}{c}{$p=500$}     & \multicolumn{2}{c}{$p=750$}     \\ \hline
			$\tau^0$ & \textbf{Coverage} & \textbf{SE} & \textbf{Coverage} & \textbf{SE} & \textbf{Coverage} & \textbf{SE} \\ \hline
			0.2      & 0.950             & 0.161       & 0.932             & 0.164       & 0.950             & 0.161       \\
			0.4      & 0.966             & 0.179       & 0.954             & 0.176       & 0.966             & 0.179       \\
			0.6      & 0.944             & 0.177       & 0.940             & 0.176       & 0.944             & 0.177       \\
			0.8      & 0.926             & 0.161       & 0.936             & 0.163       & 0.926             & 0.161       \\ \hline	
		\end{tabular}}
	\label{tab:inference}
		\caption{\footnotesize{Results of Simulation B: coverage and standard error of the $PLS$ estimator. Here, coverage (relative frequency of the number of times $\tau^0$ lies in the confidence interval), and standard error (SE, average over replications of the computed standard error of $T\tilde\tau,$ i.e, $\hat\si^2/\hat\xi^2$)}}
	
\end{table}
Partial results of Simulation A are provided in Table \ref{tab:estimation} and Table \ref{tab:detection}, the results of all remaining cases of this simulation are
provided in Table \ref{tab:estimation.app1} - Table \ref{tab:estimation.app3} in Appendix D of the supplementary materials. The results of Simulation B are provided in Table \ref{tab:inference}. The numerical findings support our theoretical results regarding detection and estimation consistency and limiting distribution behavior of the proposed methods. In terms of estimation performance from Table \ref{tab:estimation}, although the method $AL1$ clearly exhibits improving performance with increasing $T,$ the proposed method $PLS$ and the benchmark $WS$ provide nearly uniformly better result in both bias and RMSE. This is not particularly surprising, since the near optimal rate of convergence of $AL1$ derived in Theorem \ref{thm:alg1cp} is indeed slower than that of $WS$ and the optimal rate of $PLS.$ There does not appear to be a uniform distinction amongst the proposed $PLS$ and the benchmarking $WS$ method, although the proposed $PLS$ method does seem to provide a lower bias and RMSE for a large proportion of the cases considered. The detection results of Table \ref{tab:detection} bring out the important benefit of using $AL1$ in place of $WS$ as a nuisance estimation method, since the latter does not posses the ability to detect the case of $\tau^0=1.$ In all cases for $T=225,350,$ perfect detection of the change point in terms of both TPR and TNR is observed. However we do remark here that it is inevitable that TPR shall suffer when the change point moves closer to the boundary of $(0,1).$ Finally, from the coverage results of Table \ref{tab:inference}, the proposed $PLS$ method provides good control on the nominal significance level and is in keeping with the limiting distribution result of Theorem \ref{thm:limitingdist}. Furthermore, the standard error estimates appear to be stable accross increasing values of $p.$

\section*{Supplementary material}\label{SM}

This supplementary material provides four appendices. Appendix A provides the proofs to the results of Section \ref{sec:mainresults} and Section \ref{sec:nuisance} of the main article. Appendix B provides necessary stochastic bounds that are utilized in the proofs of Appendix A. Appendix C provides some auxiliary results from the literature that have been utilized in proofs of this article. Finally Appendix D provides a detailed discussion of the initializing Condition D of Algorithm 1, this appendix also provides additional details and numerical results which were omitted from Section \ref{sec:numerical} of the main article.

\appendix

\setcounter{equation}{0}
\renewcommand{\theequation}{A.\arabic{equation}}

\renewcommand{\thesection}{A}
\section*{Appendix A: Proofs}

\subsection*{Proofs of results in Section \ref{sec:mainresults}}

\begin{proof}[Proof of Lemma \ref{lem:mainlowerb}]
	A couple of observations utilized in the arguments to follow. Observe the following algebraic expansion for any $t\ge \tau^0,$
	\benr\label{eq:eq1}
	\h z_t-\h\theta_2=\h\eta^T\vep_t-\h\eta^T(\h\mu_2-\mu_2^0).
	\eenr
	Also, notice that  $\h\theta_1-\h\theta_2=\|\h\mu_1-\h\mu_2\|_2^2,$ and that the following bound that hold with probability $1-\Delta_T,$
	\benr\label{eq:ulb}
	\Big|(\h\theta_1-\h\theta_2)^2+2\h\eta^T(\h\mu_2-\mu_2^0)(\h\theta_1-\h\theta_2)\Big|\ge c_u\xi^4\Big(1-\frac{\|\h\mu_2-\mu_2^0\|_2}{\xi}\Big)\ge c_u\xi^4.
	\eenr
	This bound is obtained by using Condition A and Condition C along with the bound $c_{u1}\xi^2\le (\h\theta_1-\h\theta_2)\le c_{u2}\xi^2,$ which in turn also holds with probability $1-\Delta_T,$ and as a consequence of Condition A and C again. Now, without loss in generality (wlog) assume that $\tilde\tau\ge\tau^0,$ (the case of $\tilde\tau<\tau^0$ shall follow symmetrically) then,
	\benr
	\cU(\h z,\tau,\h\theta_1,\h\theta_2)&=&Q(\h z,\tau,\h\theta_1,\h\theta_2)-Q(\h z,\tau^0,\h\theta_1,\h\theta_2)\nn\\
	&=&\frac{1}{T}\sum_{t=1}^{\lfloor T\tau\rfloor}(\h z_t-\h\theta_1)^2 + \frac{1}{T}\sum_{t= \lfloor T\tau\rfloor+1}^{T}(\h z_t-\h\theta_2)^2 -\frac{1}{T}\sum_{t=1}^{\lfloor T\tau^0\rfloor}(\h z_t-\h\theta_1)^2 - \frac{1}{T}\sum_{t=\lfloor T\tau^0\rfloor+1}^{T}(\h z_t-\h\theta_2)^2  \nn\\
	&=&\frac{1}{T}\sum_{t=\lfloor T\tau^0\rfloor+1}^{\lfloor T\tau\rfloor}(\h z_t-\h\theta_1)^2 - \frac{1}{T}\sum_{t=\lfloor T\tau^0\rfloor+1}^{\lfloor T\tau\rfloor}(\h z_t-\h \theta_2)^2\nn\\
	&=&\frac{1}{T}\sum_{t=\lfloor T\tau^0\rfloor+1}^{\lfloor T\tau\rfloor} (\h\theta_1-\h\theta_2)^2- \frac{2}{T}\sum_{t=\lfloor T\tau^0\rfloor+1}^{\lfloor T\tau\rfloor}(\h z_t-\h\theta_2)(\h\theta_1-\h\theta_2)\nn\\
	&=&\frac{1}{T}\big(\lfloor T\tau\rfloor-\lfloor T\tau^0\rfloor\big)\big\{(\h\theta_1-\h\theta_2)^2+2\h\eta^T(\h\mu_2-\mu_2^0)(\h\theta_1-\h\theta_2)\big\}- \frac{2}{T}\sum_{t=\lfloor T\tau^0\rfloor+1}^{\lfloor T\tau\rfloor}\h\eta^T\vep_t(\h\theta_1-\h\theta_2)\nn\\
	&\ge& \frac{c_u\xi_T^4}{T}(\lfloor T\tau\rfloor-\lfloor T\tau^0\rfloor)-\frac{2\xi_T^2}{T}\sum_{t=\lfloor T\tau^0\rfloor+1}^{\lfloor T\tau\rfloor}\h\eta^T\vep_t\nn\\
	&\ge&  c_uv_T\xi_T^4-c_{u1}\si_{\vep}\xi^3\Big\{\Big(\frac{u_T}{T}\Big)^{\frac{1}{2}}+\Big(\frac{u_T}{T}\Big)^{\frac{1}{2}}\Big\{\frac{s\log (p\vee T)}{\xi\surd (Tl_T)}\Big\}\Big\}
	\ge c_u\xi^4\Big\{v_T-\frac{c_{u1}\si_{\vep}}{\xi}\Big(\frac{u_T}{T}\Big)^{\frac{1}{2}}\Big\}\nn
	\eenr
	with probability at least $1-\gamma-\Delta_T-c_1\exp\{-c_2\log (p\vee T)\}.$ Here the last equality follows by using (\ref{eq:eq1}). The first inequality follows by an application of (\ref{eq:ulb}). The second to last and the last inequality follows by an application of Lemma \ref{lem:stocb} and Condition A respectively. Uniformity over $\cG(u_T,v_T)$ is directly obtained since the stochastic bound of Lemma \ref{lem:stocb} holds uniformly over the same collection. Repeating a similar argument with $\tilde\tau\le \tau^0,$ yields the statement of this lemma.
\end{proof}

\begin{proof}[Proof of Theorem \ref{thm:optimalapprox}]
	For any $v_T>0,$ apply Lemma \ref{lem:mainlowerb} on the set $\cG(1,v_T)$ to obtain,
	\benr
	\inf_{\tau\in\cG(1,v_T)} \cU(\h z, \tau,\h\mu_1,\h\mu_2)\ge c_u\xi^{4}\Big\{v_T-\frac{c_{u1}\si_{\vep}}{\xi}\Big(\frac{1}{T}\Big)^{\frac{1}{2}}\Big\}\nn
	\eenr
	with probability at least $1-\gamma-\Delta_T-o(1).$ Then upon choosing $v_T=v_T^*\ge c_u\si_{\vep}\big/\xi \surd T,$ for an appropriately chosen $c_u>0,$ we have that $\inf_{\tau\in\cG(1,v_T)} \cU(\h z,\tau,\h\mu_1,\h\mu_2)>0.$ This implies that $\tilde\tau\notin\cG(1,v_T^*),$ i.e., $|\lfloor T\tilde\tau\rfloor-\lfloor T\tau^0\rfloor|\le T v_T^*,$ with probability $1-\gamma-\Delta_T-o(1).$ Now, reset $u_T=v_T^*$ and reapply Lemma \ref{lem:mainlowerb} for any $v_T>0$ to obtain,
	\benr
	\inf_{\tau\in\cG(u_T,v_T)} \cU(\h z,\tau,\h\mu_1,\h\mu_2)\ge c_u\xi^{4}\Big\{v_T-\frac{c_u\si_{\vep}}{\xi}\Big(\frac{u_T}{T}\Big)^{\frac{1}{2}}\Big\}\nn
	\eenr
	Now upon choosing,
	\benr
	v_T=v_T^*\ge c_u\Big(\frac{\si_{\vep}}{\xi}\Big)^{1+\frac{1}{2}}\Big(\frac{1}{T}\Big)^{\frac{1}{2}+\frac{1}{4}},
	\eenr
	we obtain that $\inf_{\tau\in\cG(u_T,v_T)} \cU(\h z,\tau,\h\mu_1,\h\mu_2)>0,$ with probability at least  $1-\gamma-\Delta_T-o(1).$ Consequently $\tilde\tau\notin\cG(u_T,v_T^*),$ i.e., $|\lfloor T\tilde\tau\rfloor-\lfloor T\tau^0\rfloor|\le T v_T^*.$ Note that the above recursion tightens the rate at each step. Continuing these recursions by resetting $u_T$ to the bound of the previous recursion, and applying Lemma \ref{lem:mainlowerb}, we obtain for the $m^{th}$ recursion,
	\benr
	\big|\lfloor T\tilde\tau\rfloor-\lfloor T\tau^0\rfloor\big|\le c_uT\Big(\frac{\si_{\vep}}{\xi}\Big)^{b_m}\Big(\frac{1}{T}\Big)^{a_m},\quad{\rm where}\quad a_m=\sum_{j=1}^{m}\frac{1}{2^j},\,\,\,  b_m=\sum_{j=0}^{m-1}\frac{1}{2^j} \nn
	\eenr
	Note that, despite the recursions in the above argument, the probability of the bound after every recursion is maintained to be at least $1-\gamma-\Delta_T-o(1).$ This follows since, the probability statement is arising from the stochastic bound of Lemma \ref{lem:stocb} applied recursively, and with a tighter bound at each recursion. Note that this yields a sequence of events such that each event is a proper subset of the one at the previous recursion. We also refer to Remark A.1 in \cite{kaul2019efficient} and Remark A.3 in \cite{kaul2019detection} for further details on this argument. To finish the proof, note that upon continuing the above recursions an infinite number of times we obtain $a_\iny=\sum_{j=1}^{\iny}1/2^j=1,$ and $b_{\iny}=\sum_{j=0}^{\iny}1/2^j=2,$ thus yielding the statement of this theorem.
\end{proof}

For a clearer exposition of the proof of Theorem \ref{thm:limitingdist} below, we use the following additional notation. Denote by
\benr\label{def:cUshort}
\h \cU(\tau)=\cU(\h z,\tau,\h\theta_1,\h\theta_2),\quad{\rm and}\quad\cU(\tau)=\cU(z,\tau,\theta_1^0,\theta_2^0),
\eenr
where $\cU(z,\tau,\theta_1,\theta_2)$ is as defined in (\ref{def:cU}). The proof of this theorem shall also rely on the `Argmax' theorem, see, Theorem 3.2.2 of \cite{vaart1996weak} (reproduced as Theorem \ref{thm:argmax}).

\begin{proof}[Proof of Theorem \ref{thm:limitingdist}]
	The structure of this proof is similar in spirit to the general approach typically used in the literature to prove this weak convergence, see, e.g. \cite{bai1994}, \cite{bai1997estimation}, \cite{bai2010common}, \cite{bhattacharjee2019change} among several others. However our specific setup involves a few more remainder terms that shall require a delicate analysis. Under the assumed regime of $\xi\to 0,$ recall from Remark \ref{rem:optapprox} that we have $T\xi^2(\tilde\tau-\tau^0)=O_p(1).$ It is thus sufficient to examine the behavior of $\tilde\tau,$ such that $\tilde\tau=\tau^0+rT^{-1}\xi^{-2}.$ Now in view of `Argmax' theorem (Theorem \ref{thm:argmax}), in order to prove the statement of this theorem it is sufficient to establish the following results, for any $|r|\le M,$ with $M>0,$
	\benr\label{eq:steps}
	&(i)& T\xi^{-2}\sup_{\tau\in\cG\big((|r|T^{-1}\xi^{-2}),0\big)} \big|\h\cU(\tau)-\cU(\tau)\big|=o_p(1),\quad {\rm and}\nn\nn\\
	&(ii)& T\xi^{-2}\cU(\tau^0+r\xi^{-2}T^{-1})\Rightarrow \big(|r|-2\si W(r)\big)
	\eenr
	The remainder of the proof is separated into two steps. {\bf Step 1} provides the result (i) of (\ref{eq:steps}) and {\bf Step 2}  provides the result (ii) of (\ref{eq:steps}). We prove both these steps for the case where $r\ge 0,$ the mirroring case of $r<0$ shall follow by symmetry.
	
	\vspace{1.5mm}	
	{\noi {\bf Step 1:}} We begin by defining the following,
	\benr
	R_1&=&\sum_{\lfloor T\tau^0\rfloor+1}^{\lfloor T\tau\rfloor}(\h\theta_1-\h\theta_2)^2-2\sum_{\lfloor T\tau^0\rfloor+1}^{\lfloor T\tau\rfloor}(\h z_t-\h\theta_2)(\h\theta_1-\h\theta_2)=R_{11}-2R_{12},\quad{\rm and}\nn\\
	R_2&=&\sum_{\lfloor T\tau^0\rfloor+1}^{\lfloor T\tau\rfloor}(\theta_1^0-\theta_2^0)^2-2\sum_{\lfloor T\tau^0\rfloor+1}^{\lfloor T\tau\rfloor}(z_t-\theta_2^0)(\theta_1^0-\theta_2^0)=R_{21}-2R_{22}.\nn
	\eenr
	Then we have the following algebraic expansion,
	\benr\label{eq:algebra}
	T\xi^{-2}\big(\h\cU(\tau)-\cU(\tau)\big)&=&T\xi^{-2}\Big(Q(\h z,\tau,\h\theta_1,\h\theta_2)-Q(\h z,\tau^0,\h\theta_1,\h\theta_2)\Big)\nn\\
	&&-T\xi^{-2}\Big(Q(z,\tau,\theta_1^0,\theta_2^0)-Q(z,\tau^0,\theta_1^0,\theta_2^0)\Big)\nn\\
	&=&\xi^{-2}\big(R_{1}-R_{2}\big)=\xi^{-2}\Big\{\big(R_{11}-2R_{12}\big)-\big(R_{21}-2R_{22}\big)\Big\}.
	\eenr
	In the following we provide uniform bounds on the expressions $\xi^{-2}\big|R_{11}-R_{21}\big|,$ and $\xi^{-2}\big|R_{12}-R_{22}\big|.$ First consider,
	\benr\label{eq:unifpart1}
	\sup_{\tau\in\cG\big((|r|T^{-1}\xi^{-2}),0\big)}\xi^{-2}|R_{11}-R_{21}|=\sup_{\tau\in\cG\big((|r|T^{-1}\xi^{-2}),0\big)}\xi^{-2}\Big|\sum_{\lfloor T\tau^0\rfloor+1}^{\lfloor T\tau\rfloor}(\h\theta_1-\h\theta_2)^2-\sum_{\lfloor T\tau^0\rfloor+1}^{\lfloor T\tau\rfloor}(\theta_1^0-\theta_2^0)^2\Big|\nn\\
	=\sup_{\tau\in\cG\big((|r|T^{-1}\xi^{-2}),0\big)}\xi^{-2}\Big|\sum_{\lfloor T\tau^0\rfloor+1}^{\lfloor T\tau\rfloor}\big\{(\h\theta_1-\h\theta_2)-(\theta_1^0-\theta_2^0)\big\}\big\{(\h\theta_1-\h\theta_2)+(\theta_1^0-\theta_2^0)\big\}\Big|\nn\\
	\le c_u\sup_{\tau\in\cG\big((|r|T^{-1}\xi^{-2}),0\big)}\Big|\sum_{\lfloor T\tau^0\rfloor+1}^{\lfloor T\tau\rfloor}\big\{(\h\theta_1-\h\theta_2)-(\theta_1^0-\theta_2^0)\big\}\Big|\hspace{1.6in}\nn\\
	\le c_u\si_{\vep}\xi\big(r\xi^{-2}\big)\Big\{\frac{s\log (p\vee T)}{Tl_T}\Big\}^{\frac{1}{2}}=o(1)\hspace{2.6in}.
	\eenr
	Here the second to last inequality follows by using the bound $(\h\theta_1-\h\theta_2)\le c_u\xi^2,$ which holds with probability at least $1-\Delta_T,$ using Condition A and Condition C. The final inequality follows using the bound $\big|(\h\theta_1-\h\theta_2)-(\theta_1^0-\h\theta_2^0)\big|\le \xi\surd\big\{s\log (p\vee T)\big/Tl_T\big\},$ that holds with probability at least $1-\Delta_T,$ again from Condition A and Condition C. The final equality holds by an application of Condition A(iii) and by using $|r|\le M.$
	
	Next consider the term $\xi^{-2}(R_{12}-R_{22}).$ An algebraic rearrangement on this difference together with an application of the elementary triangle inequality on absolute values  yields,
	\benr
	\sup_{\tau\in\cG\big((|r|T^{-1}\xi^{-2}),0\big)}\xi^{-2}|R_{12}-R_{22}|&\le& \sup_{\tau\in\cG\big((|r|T^{-1}\xi^{-2}),0\big)}\xi^{-2}\Big|\sum_{\lfloor T\tau^0\rfloor+1}^{\lfloor T\tau\rfloor}\psi_t\{(\h\theta_1-\h\theta_2)-(\theta_1^0-\theta_2^0)\}\Big|\nn\\
	&&+\sup_{\tau\in\cG\big((|r|T^{-1}\xi^{-2}),0\big)}\xi^{-2}\Big|\sum_{\lfloor T\tau^0\rfloor+1}^{\lfloor T\tau\rfloor}\{(\h\mu_1-\h\mu_2)-(\mu_1^0-\mu_2^0)\}^T\vep_t(\h\theta_1-\h\theta_2)\Big|\nn\\
	&&+\sup_{\tau\in\cG\big((|r|T^{-1}\xi^{-2}),0\big)}\xi^{-2}\Big|\sum_{\lfloor T\tau^0\rfloor+1}^{\lfloor T\tau\rfloor}(\h\mu_1-\h\mu_2)^T(\h\mu_2-\mu_2^0)(\h\theta_1-\h\theta_2)\Big|\nn\\
	&=&T1+T2+T3\nn
	\eenr
	
	The term $T3$ can be bounded above by $T3\le c_u\si_{\vep}(r\xi^{-2})\xi\surd\big\{s\log (p\vee T)/Tl_T\big\}=o(1),$ with probability at least $1-\Delta_T.$ This is achieved by using the bounds $(\h\theta_1-\h\theta_2)\le c_u\xi^2,$ and the Cauchy-Schwartz inequality on the inner product  $(\h\mu_1-\h\mu_2)^T(\h\mu_2-\mu_2^0).$ 	Term $T2$ can be bounded as given in (\ref{eq:in1}) in the proof of Lemma \ref{lem:stocb}. Upon combining this bound with the assumption (\ref{asm:jumpnew}) yields, $T2\le c_u\si_{\vep}^2(\surd r)\big\{s\log (p\vee T)\big\}\big/\big\{\xi\surd {(Tl_T)}\big\}=o(1),$ with probability at least $1-\Delta_T-o(1).$ Term $T1$ can be bounded above by utilizing the bound $\big|(\h\theta_1-\h\theta_2)-(\theta_1^0-\h\theta_2^0)\big|\le c_u\si_{\vep}\xi\surd\big\{{s\log (p\vee T)\big/Tl_T}\big\},$ together with the fundamental subgaussian bound on $|\sum\psi_t|,$ that holds with probability at least $1-\Delta_T-o(1).$ This yields $T1\le c_u\si_{\vep}^2\xi^{-1}\surd\big(r s\log (p\vee T)\big/Tl_T\big)=o(1),$ with probability at least $1-\Delta_T-o(1).$ Combining these bounds for $T1,T2$ and $T3,$ we obtain a $o(1)$ uniform bound for the term $\xi^{-2}|R_{21}-R_{22}|.$ Substituting this result together with the bound (\ref{eq:unifpart1}) in (\ref{eq:algebra}), we obtain the assertion made in Part (i) of (\ref{eq:steps}) for $r\ge 0.$ Repeating similar arguments for $r<0,$ yields the same bound and completes the proof of (i) of (\ref{eq:steps}).
	
	\vspace{1.5mm}	
	{\noi {\bf Step 2:}} Here we show that when $r\ge 0,$ we have $T\xi^{-2}\cU(\tau^0+r\xi^{-2}T^{-1})\Rightarrow \big(r-2\si W(r)\big).$ Consider,
	\benr\label{eq:weakconv}
	T\xi^{-2}\cU(\tau^0+r\xi^{-2}T^{-1})&=&T\xi^{-2}Q(z,\tau^0+r\xi^{-2}T^{-1},\theta_1^0,\theta_2^0)-T\xi^{-2}Q(z,\tau^0,\theta_1^0,\theta_2^0)\nn\\
	&=&\xi^{-2}\sum_{t=\lfloor T\tau^0\rfloor+1}^{\lfloor T\tau^0+r\xi^{-2}\rfloor}(z_t-\theta_1^0)^2-\xi^{-2}\sum_{t=\lfloor T\tau^0\rfloor+1}^{\lfloor T\tau^0+r\xi^{-2}\rfloor}(z_t-\theta_2^0)^2\nn\\
	&=&\sum_{t=\lfloor T\tau^0\rfloor+1}^{\lfloor T\tau^0+r\xi^{-2}\rfloor}\xi^2-2\sum_{t=\lfloor T\tau^0\rfloor+1}^{\lfloor T\tau^0+r\xi^{-2}\rfloor}\psi_t=T1-2\,T2
	\eenr
	For notational simplicity assume that $T\tau^0,$ and $r\xi^{-2}$ are integers, else one may resort to the inequality $\big(T(\tau-\tau^0)-1\big)\le (\lfloor T\tau\rfloor- \lfloor T\tau^0\rfloor)\le \big(T(\tau-\tau^0)+1\big),$ and show that the remainder is $o(1)$ under the assumption $\xi\to 0.$ Clearly, the term $T1=r,$ and term $T2$ can be expressed as,
	\benr
	T2=\xi\sum_{T\tau^0+1}^{T\tau^0+r\xi^{-2}}\psi_t^*,\nn
	\eenr
	where $\psi_t^*=\psi_t/\xi.$ By the definition of $\psi_t,$ we also have that $\psi_t^*,$ $t=1,...,T$ are i.i.d. mean zero subgaussian r.v.'s with variance term given by, ${\rm var(\psi_t^*)}=\eta^{0T}\Sigma_{\vep}\eta^0\big/\xi^2.$ Additionally recall by assumption we also have that  $\eta^{0T}\Sigma_{\vep}\eta^0\big/\xi^2\to \sigma.$ Furthermore,
	\benr
	T2=\xi\sum_{t=T\tau^0+1}^{T\tau^0+r\xi^{-2}}\psi_t^*= \xi\sum_{t=1}^{r\xi^{-2}}\psi_{T\tau^0+t}^*
	\eenr
	where the final equality follows by a change of index. The final term is now in a familiar form whose weak limit under $\xi\to 0$ is well known, see, e.g. Theorem 5.5  of \cite{hallmartingale} or (9) of \cite{bai1994}. In particular we have $\xi\sum_{t=1}^{r\xi^{-2}}\psi_{T\tau^0+t}^*\Rightarrow \si W_1(r),$ where $W_1(\cdot)$ is a Brownian motion on $[0,\iny).$ This completes the proof of Step 2.  Repeating similar arguments for $r<0,$ yields (ii) of (\ref{eq:steps}) and thus concludes the proof of this theorem.
\end{proof}

\subsection*{Proofs of Section \ref{sec:nuisance}}

\begin{proof}[Proof of Theorem \ref{thm:unifmean}] Although this result can be proved directly using the properties of the soft-thresholding operator $k_{\la}(\cdotp),$ by building uniform versions of arguments such as those in \cite{rothman2009generalized}, or \cite{kaul2017structural}. Instead, we provide an alternative and more illustrative proof directly using the construction (\ref{est:softthresh}).
	
	We begin by first proving Part (ii) of this Theorem, i.e., for the case where $\tau^0<1.$ For any $\tau\in\cG(u_T,0),$ an algebraic rearrangement of the elementary inequality $\big\|\bar y_{(0:\tau]}-\h\mu_1(\tau)\big\|^2+\la_1\|\h\mu_1(\tau)\|_1\le \big\|\bar y_{(0:\tau]}-\mu_1^0\big\|^2+\la_1\|\mu_1^0\|_1$ yields, $\|\h\mu_1(\tau)-\mu_1^0\|_2^2+\la_1\|\h\mu_1(\tau)\|_1\le \la_1\|\mu^0_1\|_1+ 2\big\|\bar y_{(0:\tau]}-\mu_1^0\big\|_{\iny}\big\|\h\mu_1(\tau)-\mu_1^0\big\|_1.$ Let
	\benr
	\la=c_u\max\Big[\si_{\vep}\Big\{\frac{\log (p\vee T)}{Tl_T}\Big\}^{1/2},\,\,\frac{\xi_{\iny} u_T}{l_t}\Big],\nn
	\eenr
	then applying Lemma \ref{lem:yclowerb} we obtain uniformly over $\tau\in\cG(u_T,0),$
	\benr\label{eq:elementary}
	\|\h\mu_1(\tau)-\mu_1^0\|_2^2+\la_1\|\h\mu_1(\tau)\|_1\le \la_1\|\mu^0_1\|_1+\la\big\|\h\mu_1(\tau)-\mu_1^0\big\|_1,
	\eenr
	with probability at least $1-c_{u1}\exp\{-c_{u2}\log (p\vee T)\}.$ Choosing $\la_1\ge 2\la,$ leads to $\|\big(\h\mu_1(\tau)\big)_{S^c}\|_1\le 3\|\big(\h\mu_1(\tau)-\mu_1^0\big)_{S}\|_1,$ which proves the first part of this theorem. From inequality (\ref{eq:elementary}) we also have that,
	\benr
	\|\h\mu_1(\tau)-\mu_1^0\|_2^2\le 3\la_1\|\h\mu_1(\tau)-\mu_1^0\|_1\le 3\la_1\surd{s}\|\h\mu_1(\tau)-\mu_1^0\|_2
	\eenr
	This directly implies that  $\|\h\mu_1(\tau)-\mu_1^0\|_2\le 3\la_1\surd{s}.$ To obtain the corresponding $\ell_1$ bound, note that the relation $\|\big(\h\mu_1(\tau)\big)_{S^c}\|_1\le 3\|\big(\h\mu_1(\tau)-\mu_1^0\big)_{S}\|_1$ also implies that $\|\h\mu_1(\tau)-\mu_1^0\|_1\le c_u\sqrt{s}\|\h\mu_1(\tau)-\mu_1^0\|_2.$ To finish the proof of this part recall that the only stochastic bound used here is the uniform bound over $\cG(u_T,0)$ of Lemma \ref{lem:yclowerb}, consequently the final bound also holds uniformly over the same collection. Part (i) of this Theorem, i.e. for the case where $\tau^0=1,$ can be proved by nearly identical arguments. The only change is the choice of $\la_1,$ and this arises due to the following observation. Note that, in this case we have the bound,
	\benr
	\sup_{\substack{\tau\in(0,1)\\ \tau\wedge(1-\tau)\ge l_T}}\big\|\bar y_{(0:\tau]}-\mu_1^0\big\|_{\iny}\le \la=c_u\si_{\vep}\Big\{\frac{\log (p\vee T)}{Tl_T}\Big\}^{\frac{1}{2}},\nn
	\eenr
	with probability at least $1-c_{u1}\exp\{-c_{u2}\log (p\vee T)\},$ from (ii) of Lemma \ref{lem:yclowerb}.
\end{proof}

\begin{proof}[Proof of Corollary \ref{cor:meanws}] Note that we have by Theorem 1 of \cite{wang2018high} that,
	\benr\label{eq:wsbound}
	\big|\lfloor T\h\tau\rfloor-\lfloor T\tau^0\rfloor\big|\le c_u\xi^{-2}\si^2_{\vep}\log (\log T),
	\eenr
	with probability at least $1-o(1),$ i.e., $\h\tau\in\cG(u_T,0),$ with $u_T=c_u\xi^{-2}\si^2\log \{\log T\}\big/T,$ with the same probability. Combining the bound (\ref{eq:wsbound}) with the assumption $\tau^0\wedge(1-\tau^0)\ge l_T$ and (\ref{asm:wsasm}) we have that $\h\tau\wedge(1-\h\tau)\ge c_ul_T.$ Now applying Theorem \ref{thm:unifmean} with the given choice of $u_T$ yields the following results for $\h\mu_1=\h\mu_1(\h\tau).$ First, $(\h\mu_1-\mu_1^0)\in \cA,$ with probability at least $1-o(1),$ and that
	\benr
	\|\h\mu_1-\mu_1^0\|_2\le c_u s^{\frac{1}{2}}\max\Big[\si_{\vep}\Big\{\frac{\log (p \vee T)}{Tl_T}\Big\}^{\frac{1}{2}},\,\,\frac{\xi_{\iny}\si_{\vep}^2}{\xi^2 l_T}\Big\{\frac{\log (\log T)}{T}\Big\}\Big]\le c_u\si_{\vep} \Big\{\frac{s\log (p\vee T)}{Tl_T}\Big\}^{\frac{1}{2}}\nn
	\eenr
	with probability at least $1-o(1).$ The corresponding results for $\h\mu_2$ can obtained by using similar arguments. This completes the proof of this corollary.
\end{proof}

The overall structure of the proof of Theorem \ref{thm:alg1cp} below is similar to that of Theorem \ref{thm:optimalapprox}, however to present this proof we require the following additional notation and a preliminary lemma. Recall the p-dimensional version of the least squares loss $Q$ from (\ref{def:Qpdim}) and define the following for any $\mu_1,\mu_2\in\R^p,$ $\tau\in(0,1],$ and $\mu>0,$
\benr
\cU(y,\tau,\mu_1,\mu_2)&=&Q(y,\tau,\mu_1,\mu_2)-Q(y,\tau^0,\mu_1,\mu_2),\nn\\
\cU^*(y,\tau,\mu_1,\mu_2)&=&\cU(y,\tau,\mu_1,\mu_2)+\mu\big(\|\tau\|_0^*-\|\tau^0\|_0^*\big)\nn
\eenr
Additionally, let $l_T$ be as defined in Condition A1, and for any non-negative sequence $u_T$ define the function,
\benr
F(u_T)=\begin{cases} 0 & {\rm if}\,\, u_T\big/l_T\to 0\\ 1 & {\rm otherwise}
\end{cases}\nn
\eenr
Under these notations we have the following uniform lower bound, which is essentially a version of Lemma \ref{lem:mainlowerb} in the $p$-dimensional setup. Versions of this result have also been presented in \cite{kaul2019efficient} and \cite{kaul2019detection} in a high dimensional multi-phase linear regression setting with a single and multiple change points respectively.

\begin{lem}\label{lem:lowerbpdim} Suppose the conditions of Theorem \ref{thm:alg1cp}. Let $u_T,$ and $v_T$ be any non-negative sequences and let $\cG(u_T,v_T)$ be as defined in (\ref{def:setcG}). Additionally let $\check\mu_1,$ and $\check\mu_2$ be the mean estimates of Step 0 of Algorithm 1. Then we have the following lower bounds.\\
	(i) When $\tau^0=1,$
	\benr
	\inf_{\tau\in(0,1)}\cU^*(y,\tau,\check\mu_1,\check\mu_2)\ge\mu-c_u\si_{\vep}^2\Big\{\frac{s\log (p\vee T)}{T l_T}\Big\},\nn
	\eenr
	with probability at least $1-c_{u1}\exp\{-c_{u2}\log (p\vee T)\}.$\\
	(ii) When $\tau^0<1,$
	\benr
	\inf_{\tau\in\cG(u_T,v_T)}\cU^*(y,\tau,\check\mu_1,\check\mu_2)\ge c_u\xi_T^2\Big[v_T-c_{u1}\frac{\si_{\vep}}{\xi}\Big\{\frac{u_Ts\log (p\vee T)}{T}\Big\}^{\frac{1}{2}}-\frac{\mu}{\xi^2}F(u_T)\Big],\nn
	\eenr
	with probability at least $1-c_{u1}\exp\{-c_{u2}\log (p\vee T)\}.$
\end{lem}

\begin{proof}[Proof of Lemma \ref{lem:lowerbpdim}] We begin by proving Part (ii) of this lemma, i.e., for the case where $\tau^0<1.$ First note that under the given choice of $\la_1,$ and $\la_2,$ an application of Theorem \ref{thm:unifmean} yields,
	\benr\label{eq:checkmub}
	\|\check\mu_1-\mu_1^0\|_2\le r_T=c_u\surd s\max\Big[\si_{\vep}\Big\{\frac{\log p}{T l_T}\Big\}^{\frac{1}{2}},\,\frac{\xi_{\iny}\check u_T}{l_T}\Big]
	\eenr
	with probability $1-c_{u1}\exp\{-c_{u2}\log (p\vee T)\},$ and similar for $\check\mu_2.$ In this case note that the bound (\ref{eq:checkmub}) together with Condition A1 lead to the following observations that are utilized in the argument of this proof. First,
	\benr\label{eq:obs1}
	\|\check\mu_1-\check\mu_2\|_2^2\ge \xi^2(1-2r_T/\xi-r_T^2/\xi^2)\ge c_u\xi^2,
	\eenr
	with probability at least $1-c_{u1}\exp\{-c_{u2}\log (p\vee T)\}.$ Next, an application of Cauchy-Schwartz inequality yields with the same probability,
	\benr\label{eq:obs2}
	\big|(\check\mu_2-\mu_2^0)^T(\check\mu_1-\check\mu_2)\big|\le r_T(\xi+r_T).
	\eenr	
	Combining the bounds (\ref{eq:obs1}) and (\ref{eq:obs2}) we obtain that,  	
	\benr\label{eq:bound}
	\big|\|\check\mu_1-\check\mu_2\|_2^2+2(\check\mu_2-\mu_2^0)^T(\check\mu_1-\check\mu_2)\big|\ge c_u\xi^2[1-c_{u1}r_T/\xi-c_{u2}r_T^2/\xi^2]\ge c_u\xi^2,
	\eenr
	with probability at least $1-c_{u1}\exp\{-c_{u2}\log (p\vee T)\}.$ Now consider any $\tau\in \cG(u_T,v_T)$ and wlog assume that $\tau\ge \tau^0,$ (the case of $\tau<\tau^0$ shall follow symmetrically). Then,
	\benr\label{eq:Updimbound}
	\cU(y,\tau,\check\mu_1,\check\mu_2)&=&Q(y,\tau,\check\mu_1,\check\mu_2)-Q(y,\tau^0,\check\mu_1,\check\mu_2)\nn\\
	&=&\frac{1}{T}\sum_{t=1}^{\lfloor T\tau\rfloor}\big\|y_t-\check\mu_1\big\|_2^2+\frac{1}{T}\sum_{t=\lfloor T\tau\rfloor+1}^{T}\big\|y_t-\check\mu_2\big\|_2^2\nn\\
	&&-\frac{1}{T}\sum_{t=1}^{\lfloor T\tau^0\rfloor}\big\|y_t-\check\mu_1\big\|_2^2-\frac{1}{T}\sum_{t=\lfloor T\tau^0\rfloor+1}^{T}\big\|y_t-\check\mu_2\big\|_2^2\nn\\
	&=&\frac{1}{T}\sum_{t=\lfloor T\tau^0\rfloor+1}^{\lfloor T\tau\rfloor}\big\|y_t-\check\mu_1\big\|_2^2-\frac{1}{T}\sum_{t=\lfloor T\tau^0\rfloor+1}^{\lfloor T\tau\rfloor}\big\|y_t-\check\mu_2\big\|_2^2\nn \\
	&=&\frac{1}{T}\big(\lfloor T\tau\rfloor-\lfloor T\tau^0\rfloor\big)\|\check\mu_1-\check\mu_2\|_2^2-\frac{2}{T}\sum_{t=\lfloor T\tau^0\rfloor+1}^{\lfloor T\tau\rfloor}\vep_t^T(\check\mu_1-\check\mu_2)\nn\\
	&&+\frac{2}{T}\big(\lfloor T\tau\rfloor-\lfloor T\tau^0\rfloor\big)(\check\mu_2-\mu_2^0)^T(\check\mu_1-\check\mu_2)\nn\\
	&\ge& c_uv_T\xi^2-\Big\|\frac{2}{T}\sum_{t=\lfloor T\tau^0\rfloor+1}^{\lfloor T\tau\rfloor}\vep_t\Big\|_{\iny}\|\check\mu_1-\check\mu_2\|_1.\nn\\
	&\ge& c_u\xi_T^2\Big[v_T-c_{u1}\frac{\si_{\vep}}{\xi}\Big\{\frac{u_Ts\log (p\vee T)}{T}\Big\}^{\frac{1}{2}}\Big]
	\eenr
	with probability at least $1-c_{u1}\exp\{-c_{u2}\log (p\vee T)\}.$ Here the second to last inequality follows by using (\ref{eq:bound}). The final inequality follows by using Lemma \ref{lem:epunifsup} and together with the bound $\|\check\mu_1-\check\mu_2\|_1\le c_u\xi\surd s,$ which holds with the same probability and can be obtained by using the properties of $\h\mu_1,$ $\h\mu_2$ provided in Theorem \ref{thm:unifmean} and Condition A1. Finally recall by definition,
	\benr
	\cU^*(y,\tau,\mu_1,\mu_2)&=&\cU(y,\tau,\mu_1,\mu_2)+\mu\big(\|\tau\|_0^*-\|\tau^0\|_0^*\big),\nn
	\eenr
	where $\big|\|\tau\|_0^*-\|\tau^0\|_0^*\big|\le 1.$ Also in this case where $\tau^0<1,$ we have by assumption $\tau^0\wedge(1-\tau^0)\ge l_T.$ Thus when $u_T/l_T\to 0,$ then for any $\tau\in\G(u_T,0),$ we have that $\|\tau\|_0^*=\|\tau^0\|_0^*=1.$ The statement of part (ii) of this lemma is now immediate upon noting that the bound of Lemma \ref{lem:epunifsup} used to obtain the bound (\ref{eq:Updimbound}) holds uniformly over $\cG(u_T,0),$ which is a superset of $\cG(u_T,v_T).$ This completes the proof of Part (ii). The proof of Part (i), where $\tau^0=1$ is quite straightforward. Under the given choice of $\la_1$ and $\la_2$ for this case, we have from Theorem \ref{thm:unifmean} that,
	\benr\label{eq:checkmutau1}
	\|\check\mu_1-\mu_1^0\|\le c_u\si_{\vep}\Big\{\frac{s\log p\vee T}{Tl_T}\Big\}^{\frac{1}{2}}
	\eenr
	with probability at least $1-c_{u1}\exp\{-c_{u2}\log (p\vee T)\},$ and similar for $\check\mu_2.$ Since for this case by definition $\mu_2^0=\mu_1^0,$ this directly implies that
	\benr\label{eq:checkmudifftau1}
	\|\check\mu_1-\check\mu_2\|\le c_u\si_{\vep}\Big\{\frac{s\log p\vee T}{Tl_T}\Big\}^{\frac{1}{2}}
	\eenr
	with the same probability. Now proceeding similar to that in (\ref{eq:Updimbound}) we obtain,
	\benrr
	\cU(y,\tau,\check\mu_1,\check\mu_2)&=&\frac{1}{T}\big(\lfloor T\tau\rfloor-\lfloor T\tau^0\rfloor\big)\|\check\mu_1-\check\mu_2\|_2^2-\Big\|\frac{2}{T}\sum_{t=\lfloor T\tau^0\rfloor+1}^{\lfloor T\tau\rfloor}\vep_t^T\Big\|_{\iny}\big\|\check\mu_1-\check\mu_2\big\|_1\\
	&&+\frac{2}{T}\big(\lfloor T\tau\rfloor-\lfloor T\tau^0\rfloor\big)(\check\mu_2-\mu_2^0)^T(\check\mu_1-\check\mu_2)\ge - c_u\si_{\vep}^2\Big\{\frac{s\log (p\vee T)}{T l_T}\Big\},
	\eenrr
	with probability at least $1-c_{u1}\exp\{-c_{u2}\log (p\vee T)\}.$ Here the final inequality follows by an application of the Lemma \ref{lem:epunifsup} and the inequalities (\ref{eq:checkmutau1}) and (\ref{eq:checkmudifftau1}). The statement of Part (i) now follows since for any $\tau\in(0,1),$ we have $\|\tau\|_0^*=1.$  This finishes the proof of this lemma.
\end{proof}

\begin{proof}[Proof of Theorem \ref{thm:alg1cp}] We begin by proving Part (i) of this theorem, i.e., when $\tau^0=1.$ Note that we have by Part (i) of Lemma \ref{lem:lowerbpdim},
	\benr
	\inf_{\tau\in(0,1)}\cU^*(y,\tau,\check\mu_1,\check\mu_2)\ge\mu-c_u\si_{\vep}^2\Big\{\frac{s\log (p\vee T)}{T l_T}\Big\},\nn
	\eenr
	with probability at least $1-c_{u1}\exp\{-c_{u2}\log (p\vee T)\}.$ Now by choice of $\mu=c_u\si_{\vep}\xi\surd\big\{s\log (p\vee T)\big/T\big\},$ together with Condition A1, we have that $\inf_{\tau\in(0,1)}\cU^*(y,\tau,\check\mu_1,\check\mu_2)>0,$ thus implying that $\h\tau\notin(0,1).$ This leaves us with the only possibility that $\h\tau=1,$ with probability at least $1-c_{u1}\exp\{-c_{u2}\log (p\vee T)\}.$ This completes the proof of Part (i). We now proceed to the proof of Part (ii) of this theorem, i.e. for the case where $\tau^0<1.$ For this purpose, first note that using Part (ii) of Lemma \ref{lem:lowerbpdim} we have for $v_T>0$ that,
	\benrr
	\inf_{\tau\in\cG(1,v_T)}\cU^*(y,\tau,\check\mu_1,\check\mu_2)\ge c_u\xi_T^2\Big[v_T-c_{u1}\frac{\si_{\vep}}{\xi}\Big\{\frac{s\log (p\vee T)}{T}\Big\}^{\frac{1}{2}}-\frac{\mu}{\xi^2}\Big].
	\eenrr
	with probability at least $1-c_{u1}\exp\{-c_{u2}\log (p\vee T)\}.$ Upon choosing,
	\benr
	v_T=v_T^*\ge c_u\frac{\si_{\vep}}{\xi}\Big\{\frac{s\log (p\vee T)}{T}\Big\}^{\frac{1}{2}},\nn
	\eenr
	we obtain that $\inf_{\tau\in\cG(1,v_T)}\cU^*(y,\tau,\check\mu_1,\check\mu_2)>0,$ thus implying that $\h\tau\in\cG(v_T^*,0)$ with the same probability. Resetting $u_T=v_T^*$ and reapplying Part (ii) of Lemma \ref{lem:lowerbpdim} we obtain with probability at least $1-c_{u1}\exp\{-c_{u2}\log (p\vee T)\},$
	\benrr
	\inf_{\tau\in\cG(u_T,v_T)}\cU^*(y,\tau,\check\mu_1,\check\mu_2)\ge c_u\xi_T^2\Big[v_T-c_{u1}\frac{\si_{\vep}}{\xi}\Big\{u_T\frac{s\log (p\vee T)}{T}\Big\}^{\frac{1}{2}}\Big].\nn
	\eenrr
	Note that in this recursive step we have $F(u_T)=0,$ since by Condition A1 we have that $v_T^*/l_T\to 0.$ Now upon choosing
	\benr
	v_T=v_T^*\ge c_u\Big(\frac{\si_{\vep}}{\xi}\Big)^{1+\frac{1}{2}}\Big\{\frac{s\log (p\vee T)}{T}\Big\}^{\frac{1}{2}+\frac{1}{4}},\nn
	\eenr
	we obtain that $ \inf_{\tau\in\cG(u_T,v_T)}\cU^*(y,\tau,\check\mu_1,\check\mu_2)>0,$ consequently yielding $\h\tau\in\cG(v_T^*,0).$ Continuing these recursions by resetting $u_T$ to the bound of the previous recursion, we obtain for the $m^{th}$ recursion,
	\benr
	\big|\lfloor T\tau\rfloor-\lfloor T\tau^0\rfloor\big| \le c_uT\Big(\frac{\si_{\vep}}{\xi}\Big)^{b_m}\Big\{\frac{s\log (p\vee T)}{T}\Big\}^{a_m},\quad{\rm where}\quad a_m=\sum_{j=1}^{m}\frac{1}{2^j},\,\,\,  b_m=\sum_{j=0}^{m-1}\frac{1}{2^j} \nn
	\eenr
	Note that, despite the recursions in the above argument, the probability of the bound after every recursion is maintained to be at least  $1-c_{u1}\exp\{-c_{u2}\log (p\vee T)\}.$ This follows since by the same reasoning as discussed in the proof of Theorem \ref{thm:optimalapprox}. To finish the proof, note that upon continuing the above recursions an infinite number of times we obtain $a_\iny=\sum_{j=1}^{\iny}1/2^j=1,$ and $b_{\iny}=\sum_{j=0}^{\iny}1/2^j=2,$ thus yielding the statement of this theorem.
\end{proof}

\begin{proof}[Proof of Corollary \ref{cor:final}] The proof of this result is a direct consequence of Theorem \ref{thm:unifmean} and \ref{thm:alg1cp}. In particular, we have from Theorem \ref{thm:alg1cp},
	\benr
	\big|\lfloor T\tau\rfloor-\lfloor T\tau^0\rfloor\big| \le c_u\si_{\vep}^2\xi^{-2}s\log (p\vee T)\nn
	\eenr
	with probability at least  $1-c_{u1}\exp\{-c_{u2}\log (p\vee T)\},$ i.e. $\h\tau\in\cG(u_T,0),$ with $u_T=_u\si_{\vep}^2\xi^{-2}s\log (p\vee T)\big/T$ with the same probability. Using this bound together with the assumption $\tau^0\wedge(1-\tau^0)$ and Condition A1 also yields that $\h\tau\wedge(1-\h\tau)\ge c_ul_T$ with the same probability. The statement of this result now follows by an application of Theorem \ref{thm:unifmean} with the given choice of $u_T$ and an application of condition (\ref{asm:additional}).
\end{proof}

\section*{Appendix B: Stochastic bounds}
\begin{lem}\label{lem:sparselinearep} Suppose $\vep_t,$ $t=1,...,T$ are i.i.d r.v.'s satisfying Condition B for any $T\ge 1.$ Let $\cK(c_u^2s)=\{\delta\in\R^p;\,\,\|\delta\|_0\le c_u^2s;\,\,\|\delta\|_2=1\}$ be subset of $\R^p,$ for $s\ge 1.$ Then we have the following uniform bound.
	\benr
	\sup_{\delta\in\cK(c_u^2s)}\Big|\frac{1}{T}\sum_{t=1}^T\delta^T\vep_t\Big|\le c_u\si_{\vep}\Big\{\frac{s\log (p\vee T)}{T}\Big\}^{\frac{1}{2}}\nn
	\eenr
	with probability at least $1-c_{u1}\exp\big\{-c_{u2}\log (p\vee T)\big\}.$
\end{lem}

\begin{proof}[Proof of Lemma \ref{lem:sparselinearep}] The arguments of this proof are essentially adopted from Lemma 15 of the Supplementary materials of \cite{loh2012}. Consider any subset $U\subseteq\{1,...,p\},$ and define the set $T_U=\{\delta\in\R^p;\,\,\|\delta\|_2\le 1,\,\,{\rm Supp}(\delta)\subseteq U\}.$ Let $\cW=\{u_1,...,u_m\}$ be a $1/3$-cover of $T_U,$ i.e., for every $\delta\in T_U,$ there is some $u_i\in\cW$ such that $\|\Delta \delta\|_2\le 1/3,$ where $\Delta=\delta-u_i.$ Note that it is well known (see, page 94 of \cite{vaart1996weak}) that we can construct $\cW$ such that $|\cW|\le 9^{c_u^2s}.$ Now consider,
	\benr
	\sup_{\delta\in T_U}\big|\sum_{t=1}^T\vep_t^T\delta\big|\le \max_{i}\big|\sum_{t=1}^T\vep_t^Tu_i\big|+ \sup_{\delta\in T_U}\max_{i}\big|\sum_{t=1}^T\vep_t^T\Delta \delta\big|\nn
	\eenr
	By construction of $\cW,$ we also have that $3\Delta \delta\in T_U,$ hence it follows that,
	\benr
	\sup_{\delta\in T_U}\big|\sum_{t=1}^T\vep_t^T\delta\big|\le \max_{i}\big|\sum_{t=1}^T\vep_t^Tu_i\big|+ \frac{1}{3}\sup_{\delta\in T_U}\big|\sum_{t=1}^T\vep_t^T\delta\big|.\nn
	\eenr
	This implies $\sup_{\delta\in T_U}\big|\sum_{t=1}^T\vep_t^T\delta\big|\le (3/2)\max_{i}\big|\sum_{t=1}^T\vep_T^Tu_i\big|.$ Now applying the fundamental subgaussian bound (Lemma \ref{lem:fsubgbound}) for each $i$ and taking a union over all $i,$ we obtain for any $\la>0,$
	\benr
	pr\Big(\sup_{\delta\in T_U}\frac{1}{T}\big|\sum_{t=1}^T\vep_t^T\delta\big|\ge \la\Big)\le 9^{c_u^2s}2\exp\Big(-\frac{c_uT\la^2}{\si_{\vep}^2}\Big)\nn
	\eenr
	Finally upon noting that $\cK(c_u^2s)=\bigcup_{|U|\le c_u^2s} T_U$ and taking a union bound over ${p\choose \lfloor c_u^2 s\rfloor} \le p^{c_u^2s}$ choices of $U$ yields,
	\benr
	pr\Big(\sup_{\delta\in \cK(c_u^2s)}\frac{1}{T}\big|\sum_{t=1}^T\vep_t^T\delta\big|\ge \la\Big)\le 2\exp\Big(-\frac{c_uT\la^2}{\si_{\vep}^2}+c_us\log p\Big)\nn
	\eenr
	The statement of this lemma now follows upon choosing $\la=c_u\si_{\vep}\surd\big\{s\log (p\vee T)\big/T\big\},$ for an appropriately chosen $c_u>0.$
\end{proof}

\begin{lem}\label{lem:unifb.restrictedset}  Suppose $\vep_t,$ $t=1,...,T$ are i.i.d r.v.'s satisfying Condition B for any $T\ge 1.$ Let $\cA^*=\{\delta\in\R^p;\,\,\|\delta\|_1\le c_u\surd{s};\,\,\|\delta\|_2=1\}$ be subset of $\R^p,$ for $s\ge 1.$ Then we have the following uniform bound.
	\benr
	\sup_{\delta\in\cA}\Big|\frac{1}{T}\sum_{t=1}^T\delta^T\vep_t\Big|\le c_u\si_{\vep}\Big\{\frac{s\log (p\vee T)}{T}\Big\}^{\frac{1}{2}}\nn
	\eenr
	with probability at least $1-c_{u1}\exp\big\{-c_{u2}\log (p\vee T)\big\}.$
\end{lem}

\begin{proof}[Proof of Lemma \ref{lem:unifb.restrictedset}]
	The arguments of this proof are essentially adopted from Lemma 12 of the Supplementary materials of \cite{loh2012}. Consider the collection $\cK(c_u^2s)=\cB_0(c_u^2s)\cap\cB_2(1),$ also defined in Lemma \ref{lem:sparselinearep}, then by Lemma \ref{lem:sparselinearep} we have that,
	\benr
	\sup_{\delta\in\cK(c_u^2s)}\Big|\frac{1}{T}\sum_{t=1}^T\delta^T\vep_t\Big|\le c_u\si_{\vep}\Big\{\frac{s\log (p\vee T)}{T}\Big\}^{\frac{1}{2}}\nn
	\eenr
	with probability at least $1-c_{u1}\exp\big\{-c_{u2}\log (p\vee T)\big\}.$ Now, by Lemma \ref{lem:lwtopo}, the desired bound over the collection $\cA^*,$ can be reduced to proving the same bound for all vectors $\delta\in 3\,{\rm conv}\big\{\cK(c_u^2s)\big\}.$ Consider any linear combination $\delta=\sum_i\alpha_i\delta_i,$ with $\al_i\ge 0,$ such that $\sum_i\alpha_i=1,$ and that $\|\delta_i\|_0\le c_u^2s$ and $\|\delta_i\|_2\le 3,$ for each $i.$ Then,
	\benr
	\frac{1}{T}\sup_{\delta\in\cA^*}\big|\sum_{t=1}^T\delta^T\vep_t\big|\le 3\sum_{i}\al_i\sup_{\delta_i\in\cK(c_u^2s)}\frac{1}{T}\big|\sum_{t=1}^T\delta_i^T\vep_t\big|\le c_u\si_{\vep}\Big\{\frac{s\log (p\vee T)}{T}\Big\}^{\frac{1}{2}}\nn
	\eenr
	with probability at least $1-c_{u1}\exp\big\{-c_{u2}\log (p\vee T)\big\}.$
\end{proof}

\begin{lem}\label{lem:stocb} Let $\h z,$ $\h\theta_2$ be as defined in Section \ref{sec:intro} and $\cG$ be as defined in (\ref{def:setcG}). Suppose Condition B and C hold and let $u_T$ be any non-negative sequence, then for any $0<\gamma<1,$ there exists $c_u>0$ such that,
	\benr
	\sup_{\substack{\tau\in\cG(u_T,0)\\ \tau\ge\tau^0}}\frac{1}{T}\Big|\sum_{t=\lfloor T\tau^0\rfloor+1}^{\lfloor T\tau\rfloor} \h\eta^T\vep_t\Big|\le c_u\si_{\vep}\Big\{\xi\Big(\frac{u_T}{T}\Big)^{\frac{1}{2}}+\Big(\frac{u_T}{T}\Big)^{\frac{1}{2}}\big(\frac{s\log (p\vee T)}{\surd Tl_T}\Big)\Big\},\nn
	\eenr
	with probability at least $1-\gamma-\Delta_T-c_{u1}\exp\big\{-c_{u2}\log (p\vee T)\big\}.$
\end{lem}

\begin{proof}[Proof of Lemma \ref{lem:stocb}]
	For any $\tau\in \cG(u_T,0),$ $\tau\ge \tau^0$ we have,
	\benr
	\frac{1}{T}\Big|\sum_{t=\lfloor T\tau^0\rfloor+1}^{\lfloor T\tau\rfloor} \h\eta^T\vep_t\Big|\le  \frac{1}{T}\Big|\sum_{t=\lfloor T\tau^0\rfloor+1}^{\lfloor T\tau\rfloor} \eta^{0T}\vep_t\Big|+ \frac{1}{T}\Big|\sum_{t=\lfloor T\tau^0\rfloor+1}^{\lfloor T\tau\rfloor} (\h\eta-\eta^0)^T\vep_t\Big|=R1+R2\nn
	\eenr	
	Using the fundamental subgaussian bound of Lemma \ref{lem:fsubgbound} we obtain that $R1\le c_u\xi\si_{\vep}\surd{\big(\lfloor T\tau\rfloor-\lfloor T\tau^0\rfloor\big)}\big/T,$ for some $c_u>0,$ with probability at least $1-\gamma.$ On the set $\cG(u_T,0),$ we also have that $\big(\lfloor T\tau\rfloor -\lfloor T\tau^0\rfloor\big)\le T u_T,$ thus,
	\benr\label{eq:in2}
	\sup_{\substack{\tau\in\cG(u_T,0)\\ \tau\ge\tau^0}}R1\le c_u\xi\si_{\vep}\Big(\frac{u_T}{T}\Big)^{\frac{1}{2}}
	\eenr	
	with probability at least $1-\gamma.$ Next consider term $R2,$
	\benr\label{eq:in1}
	\frac{1}{T}\Big|\sum_{t=\lfloor T\tau^0\rfloor+1}^{\lfloor T\tau\rfloor} (\h\eta-\eta^0)^T\vep_t\Big|\le \frac{1}{T}\Big|\sum_{t=\lfloor T\tau^0\rfloor+1}^{\lfloor T\tau\rfloor} (\h\mu_1-\mu_1^0)^T\vep_t\Big|+\frac{1}{T}\Big|\sum_{t=\lfloor T\tau^0\rfloor+1}^{\lfloor T\tau\rfloor} (\h\mu_2-\mu_2^0)^T\vep_t\Big|
	\eenr
	By Condition C we have that $(\h\mu_1-\mu_1^0)\in\cA,$ which directly implies that $\delta=(\h\mu_1-\mu_1^0)\big/\|\h\mu_1-\mu_1^0\|_2\in\cA^*,$ where $\cA^*$ is defined in Lemma \ref{lem:unifb.restrictedset}. Thus an application of Lemma \ref{lem:unifb.restrictedset} provides a the following bound on the first term in the rhs of (\ref{eq:in1}).
	\benr
	\frac{1}{T}\Big|\sum_{t=\lfloor T\tau^0\rfloor+1}^{\lfloor T\tau\rfloor} (\h\mu_1-\mu_1^0)^T\vep_t\Big|\le c_u\|\h\mu_1-\mu_1^0\|_2\si_{\vep}\surd(s \log (p\vee T))\frac{\surd(\lfloor T\tau\rfloor -\lfloor T\tau^0\rfloor)}{T},\nn
	\eenr
	that holds with probability at least $1-c_{u1}\exp\big\{-c_{u2}\log (p\vee T)\big\}.$ The same bound argument also applies to the second term in the rhs of (\ref{eq:in1}). Finally, using the rate assumption of Condition C and the inequality $\big(\lfloor T\tau\rfloor -\lfloor T\tau^0\rfloor\big)\le T u_T,$ on the set $\cG(u_T,0)$ we obtain that,
	\benr\label{eq:in3}
	\sup_{\substack{\tau\in\cG(u_T,0)\\ \tau\ge\tau^0}}\frac{1}{T}\Big|\sum_{t=\lfloor T\tau^0\rfloor+1}^{\lfloor T\tau\rfloor} (\h\eta-\eta^0)^T\vep_t\Big|\le c_u\si_{\vep}\Big(\frac{u_T}{T}\Big)^{\frac{1}{2}}\Big(\frac{s\log (p\vee T)}{\surd Tl_T}\Big),
	\eenr
	with probability at least $1-\Delta_T-c_{u1}\exp\big\{-c_{u2}\log (p\vee T)\big\}.$ The statement of this lemma follows by combining the bounds (\ref{eq:in2}) and (\ref{eq:in3}).
\end{proof}


\begin{lem}\label{lem:epunifsup}  Suppose $\vep_t,$ $t=1,...,T$ are i.i.d r.v.'s satisfying Condition B for any $T\ge 1.$ Then,
	\benr
	\sup_{\substack{\tau\in\cG(u_T,0)\\\tau\ge \tau^0}}\frac{1}{T}\Big\|\sum_{t=\lfloor T\tau^0\rfloor+1}^{\lfloor T\tau\rfloor}\vep_t\Big\|_{\iny}\le c_u\si_{\vep}\Big\{\frac{u_T\log (p\vee T)}{T}\Big\}^{\frac{1}{2}},
	\eenr	
	with probability at least $1-c_{u1}\exp\big\{-c_{u2}\log (p\vee T)\big\}.$
\end{lem}

\begin{proof}[Proof of Lemma \ref{lem:epunifsup}] Let $\delta_j\in\R^p$ be the unit vector in the $j^{th}$ direction, i.e., $\delta_{jk}=1,$ $k=j$ and $\delta_{jk}=0,$ $k\ne j.$ Then applying the fundamental subgaussian bound of Lemma \ref{lem:fsubgbound} we obtain,
	\benr
	\frac{1}{T}\Big\|\sum_{t=\lfloor T\tau^0\rfloor+1}^{\lfloor T\tau\rfloor}\delta_j^T\vep_t\Big\|_{\iny}\le c_u\si_{\vep}\surd(\log (p\vee T))\frac{\surd(\lfloor T\tau\rfloor -\lfloor T\tau^0\rfloor)}{T}\nn
	\eenr
	with probability at least $1-c_{u1}\exp\big\{-c_{u2}\log (p\vee T)\big\}.$ Taking a union bound over $j=1,...,p$ yields
	\benr
	\frac{1}{T}\Big\|\sum_{t=\lfloor T\tau^0\rfloor+1}^{\lfloor T\tau\rfloor}\vep_t\Big\|_{\iny}\le \max_j 	\frac{1}{T}\Big\|\sum_{t=\lfloor T\tau^0\rfloor+1}^{\lfloor T\tau\rfloor}\delta_j^T\vep_t\Big\|_{\iny}\le c_u\si_{\vep}\surd(\log (p\vee T))\frac{\surd(\lfloor T\tau\rfloor -\lfloor T\tau^0\rfloor)}{T}\nn
	\eenr
	with probability at least $1-c_{u1}\exp\big\{-c_{u2}\log (p\vee T)\big\}.$ Finally using the relation $\big(\lfloor T\tau\rfloor -\lfloor T\tau^0\rfloor\big)\le T u_T,$ on the set $\cG(u_T,0)$ we obtain that,
	\benr
	\sup_{\substack{\tau\in\cG(u_T,0)\\\tau\ge \tau^0}}\frac{1}{T}\Big\|\sum_{t=\lfloor T\tau^0\rfloor+1}^{\lfloor T\tau\rfloor}\vep_t\Big\|_{\iny}\le  c_u\si_{\vep}\Big\{\frac{u_T\log (p\vee T)}{T}\Big\}^{\frac{1}{2}},\nn
	\eenr	
	with probability at least $1-c_{u1}\exp\big\{-c_{u2}\log (p\vee T)\big\}.$
\end{proof}

\begin{lem}\label{lem:yclowerb} Suppose Condition B and let $\bar y_{(0:\tau]}$ and $\bar y_{(\tau:1]}$ be as defined in (\ref{def:empmeans}) and assume that $Tl_T\ge c_u,$ for an appropriately chosen $c_u.$ Additionally let $\|\mu_1-\mu_2\|_{\iny}\le\xi_{\iny},$ then,\\
	(i) when $\tau^0=1$ we have,
	\benr
	\sup_{\substack {\tau\in(0,1)\\ \tau\wedge(1-\tau)\ge c_ul_T}}\big\|\bar y_{(0:\tau]}-\mu_1^0\big\|_{\iny}\le c_u\si_{\vep}\Big\{\frac{\log (p\vee T)}{Tl_T}\Big\}^{\frac{1}{2}}\nn
	\eenr
	with probability at least $1-c_{u1}\exp\big\{-c_{u2}\log (p\vee T)\big\}.$\\
	(ii) when $\tau^0<1$ we have for any non-negative $u_T,$
	\benr
	\sup_{\substack{\tau\in\cG(u_T,0)\\ \tau\wedge(1-\tau)\ge c_ul_T}}\big\|\bar y_{(0:\tau]}-\mu_1^0\big\|_{\iny}\le c_u\max\Big[\si_{\vep}\Big\{\frac{\log (p\vee T)}{Tl_T}\Big\}^{\frac{1}{2}},\,\,\frac{u_T\xi_{\iny}}{l_T}\Big]\nn
	\eenr
	with probability at least $1-c_{u1}\exp\big\{-c_{u2}\log (p\vee T)\big\}.$ The same uniform upper bounds also hold for $\big\|\bar y_{(\tau:1]}-\mu_2^0\big\|_{\iny},$ where for the case $\tau^0=1,$ define $\mu_2^0=\mu_1^0.$
\end{lem}

\begin{proof} We begin by proving Part (i) of this lemma. When $\tau^0=1,$ note that, $\big(\bar y_{(0:\tau]}-\mu_1^0\big)=\sum_{t=1}^{\lfloor T\tau\rfloor}\vep_t\big/\lfloor T\tau\rfloor.$ Thus applying the fundamental subgaussian bound of Lemma \ref{lem:fsubgbound} together with a union over $p$ projections (as done in the proof of Lemma \ref{lem:epunifsup}) we have,
	\benr
	\big\|\bar y_{(0:\tau]}-\mu_1^0\big\|_{\iny}\le c_u\si_{\vep}\Big\{\frac{\log (p\vee T)}{\lfloor T\tau\rfloor}\Big\}^{\frac{1}{2}}\nn
	\eenr
	with probability at least $1-c_{u1}\exp\big\{-c_{u2}\log (p\vee T)\big\}.$ The uniform bound of Part (i) follows by using the restriction $\tau\wedge(1-\tau)\ge c_ul_T,$ and $Tl_T\ge c_u.$ Next we proceed to the proof of Part (ii). Note that for any $\tau\in(0,1),$
	\benr
	\big\|\bar y_{(0:\tau]}-\mu_1^0\big\|_{\iny}\le \frac{1}{\lfloor T\tau\rfloor}\big\|\sum_{t=1}^{\lfloor T\tau\rfloor}\vep_t\big\|_{\iny}+\frac{\big|\lfloor T\tau\rfloor-\lfloor T\tau^0\rfloor\big|}{\lfloor T\tau\rfloor}\big\|\mu_1^0-\mu_2^0\big\|_{\iny}=R1+R2\nn
	\eenr
	By arguments used to prove Part (i) we have that,
	\benr
	\sup_{\substack {\tau\in(0,1)\\ \tau\wedge(1-\tau)\ge c_ul_T}} R1\le c_u\si_{\vep}\Big\{\frac{\log (p\vee T)}{\lfloor T\tau\rfloor}\Big\}^{\frac{1}{2}}
	\eenr
	with probability at least $1-c_{u1}\exp\big\{-c_{u2}\log (p\vee T)\big\}.$ To uniformly bound $R2,$ first note that $\|\mu_1^0-\mu_2^0\big\|_{\iny}\le \|\mu_1^0-\mu_2^0\|_2.$ Using this inequality together with the restrictions $\big|\lfloor T\tau\rfloor-\lfloor T\tau^0\rfloor\big|\le Tu_T$ that holds on the set $\cG(u_T,0),$ and $\tau\wedge(1-\tau)\ge c_ul_T,$ we obtain that
	\benr
	\sup_{\substack{\tau\in\cG(u_T,0)\\ \tau\wedge(1-\tau)\ge c_ul_T}} R2\le c_u\frac{u_T\xi_{\iny}}{l_T}\nn
	\eenr
	The statement of Part (ii) of this lemma follows by combining these uniform bounds for $R1$ and $R2.$
\end{proof}

\section*{Appendix C: Auxiliary results}

The following lemma is the fundamental subgaussian tail bound, and has been reproduced from Lemma 1.3 of \cite{rigollet201518}.
\begin{lem}\label{lem:fsubgbound} Let $X$ be any subgaussian$(\si^2)$ random variable. Then for any $t>0,$ it holds
	\benr
	pr\big(|X|>t\big)\le \exp\Big(-\frac{t^2}{2\si^2}\Big)\nn
	\eenr
\end{lem}

The following lemma is essentially Lemma 11 of the Supplementary materials of \cite{loh2012}.
\begin{lem}\label{lem:lwtopo}  For any $s\ge 1,$ we have
	\benr
	\cB_1(c_u\surd{s})\cap\cB_2(1)\subseteq 3{\rm cl}\Big[{\rm conv}\big\{\cB_0(c_u^2s)\cap\cB_2(1)\big\}\Big],
	\eenr
	where the balls are taken in $p$-dimensional space, and ${\rm cl}(\cdot)$ and ${\rm conv}(\cdot)$ denote the topological closure and convex hull, respectively.	
\end{lem}

\begin{proof}[Proof of Lemma \ref{lem:lwtopo}] The argument of this proof is nearly identical to that of Lemma 11 in \cite{loh2012}. The desired containment is trivial when $s>p,$ hence assume that $1\le s\le p.$ For any closed and convex sets $A$ and $B$ and support function $\phi_A(z)=\sup_{\delta\in A}\langle\delta,z\rangle,$ $z\in\R^p$, and similar $\Phi_B(\cdotp),$ it is known that (Theorem 2.3.1(c) of \cite{hug2010course}) $\phi_A\le\phi_B$ if and only if $A\subseteq B.$ The remainder of this proof verifies this relation for the sets $A=\cB_1(c_u\sqrt{s})\cap\cB_2(1)$ and $B=3{\rm cl}\Big[{\rm conv}\big\{\cB_0(c_u^2s)\cap\cB_2(1)\big\}\Big].$ For $z\in\R^p,$ let $S\subseteq\{1,2,...,p\}$ be the subset that indexes the top $\lfloor c_u^2s\rfloor$ elements of $z$ in magnitude. Then $\|z_{S^c}\|_{\iny}\le |z_j|,$ for all $j\in S,$ this in turn implies that,
	\benr
	\|z_{S^c}\|\le \frac{1}{\lfloor c_us\rfloor}\|z_{S}\|_1\le \frac{1}{\surd \lfloor c_u^2s\rfloor}\|z_{S}\|_2\nn
	\eenr
	Now observe that,
	\benr
	\phi_A(z)&=&\sup_{\delta\in A}\langle\delta,z\rangle\le \sup_{\|\delta_S\|_2\le 1}\langle\delta_S,z_S\rangle+\sup_{\|\delta_{S^c}\|_1\le c_u\surd{s}}\langle\delta_{S^c},z_{S^c}\rangle\nn\\
	&\le& \|z_S\|_2+c_u\surd{s}\|z_{S^c}\|_{\iny}\le\Big(1+\frac{ c_u\surd s}{\surd\lfloor c_u^2s\rfloor}\Big)\|z_S\|_2\le 3\|z_S\|_2\nn
	\eenr
	The statement of the lemma now follows upon noting that $\phi_B(z)=3\|z_S\|_2.$
\end{proof}

The following theorem is the well known `Argmax' theorem reproduced from Theorem 3.2.2 of \cite{vaart1996weak}
\begin{thm}[Argmax Theorem]\label{thm:argmax} Let $\cM_n,\cM$ be stochastic processes indexed by a metric space $H$ such that $\cM_n\Rightarrow\cM$ in $\ell^{\iny}(K)$ for every compact set $K\subseteq H$\footnote{i.e., $\sup_{h\in K}\big|\cM_n(h)-\cM(h)\big|\to^p 0.$}. Suppose that almost all sample paths $h\to \cM(h)$ are upper semicontinuous and posses a unique maximum at a (random) point $\h h,$ which as a random map in $H$ is tight. If the sequence $\h h_n$ is uniformly tight and satisfies $\cM_n(\h h_n)\ge \sup_h \cM_n(h)-o_p(1),$ then $\h h_n\Rightarrow \h h$ in $H.$
\end{thm}

\section*{Appendix D: Further details}

\subsection*{Discussion on Algorithm 1 and its initializing Condition D}

In this subsection we provide a detailed discussion of the initializing requirement of $\check\tau$ of Step 0 Algorithm 1 given in Condition D, with the objective of thoroughly convincing the reader of its mildness. We being with a potentially counterintuitive numerical observation which forms the basis for the construction of Condition D and the proposed Algorithm 1. Suppose the $p$-dimensional time series model (\ref{model:subgseries}), and first choose virtually any initial value $\check\tau\in (0,1),$ separated from its boundaries. Then compute the initial soft-thresholded mean estimates $\check\mu_1=\h\mu_1(\check\tau),$ $\check\mu_2=\h\mu_2(\check\tau)$ on the basis of the corresponding binary partition yielded by the arbitrary choice $\check\tau$. Clearly, $\check\mu_1,$ and $\check\mu_2$ may be very poor estimates that may be nowhere near the true values $\mu_1^0$ and $\mu_2^0$ respectively. Nevertheless, upon performing a single update (Step 1 of Algorithm 1) of the change point estimate using $\check\mu_1,$ and $\check\mu_2,$ yields a very precise estimate of the unknown change point, irrespective of the choice of the initial change point and irrespective of the location of the unknown change point. We present Figure \ref{fig:insensitivity1} below, to provide a preliminary visual impression of the robustness of this procedure to the initial value which is the motivation of Algorithm 1 and the initializing Condition D,

\vspace{-2mm}
\begin{figure}[H]
	\centering
	\resizebox{0.8\textwidth}{!}{
		\begin{minipage}[b]{0.45\textwidth}
			\includegraphics[width=\textwidth]{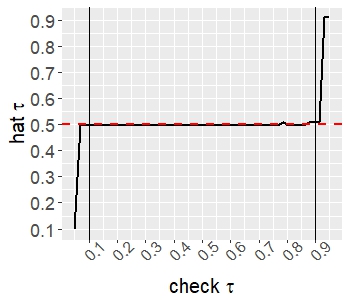}
		\end{minipage}
		\hspace{1in}
		\begin{minipage}[b]{0.45\textwidth}
			\includegraphics[width=\textwidth]{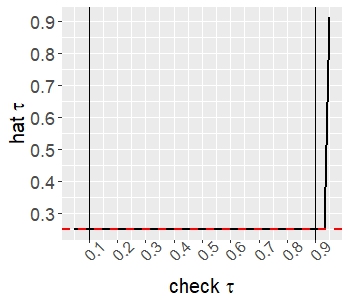}
	\end{minipage}}
	\caption{Illustration of robustness of Algorithm 1 to the initializer $\check\tau.$ x-axis: initializer $\check\tau,$ y-axis: estimated change point $\h\tau$ of Algorithm 1. This illustration is based on a single realization $y,$ with $\tau_0=0.5$ (Left panel: indicated by red line) and $\tau_0=0.25$ (Right panel:indicated by red line). Additional parameters: $T=225,$ $p=100,$ $\mu_1^0=(1_{1\times 5}, 0_{1\times p-5})^T,$ $\g_0=(0_{1\times 5},1_{1\times 5}, 0_{1\times p-10})^T$ and $\vep_t\sim^{i.i.d.}\cN(0,\Sigma),$ with $\Sigma_{ij}=\rho^{|i-j|}.$}
	\label{fig:insensitivity1}
\end{figure}

From Figure \ref{fig:insensitivity1}, note that any value of the initializer $\check\tau\in (0.1,0.9)$ (almost the entire parametric space $(0,1]$ of $\tau^0$), yields estimates $\h\tau$ which approximate $\tau^0$ with nearly identical precision. This behavior is also true irrespective of the location of $\tau^0,$ the true change-point. This goes against the natural intuition, that the `better' the initial value $\check\tau,$ the `better' is the updated estimate $\h\tau,$ in which case, one would have expected a smooth $S$ shaped transition from one end of the parametric space to the other. Instead a flat line behavior for nearly all values of $\check\tau,$ with an abrupt change at the very edges is observed.

This observation is very surprising, since it suggests that any initial $\check\tau$ which carries any `fractional amount of information' on the unknown $\tau^0,$ can be utilized to obtain an estimate $\h\tau$ which lies in a near optimal neighborhood of $\tau^0.$ In other words, the update process pulls in the initial guess $\check\tau$ from a much wider neighborhood (nearly arbitrary) of $\tau^0,$ to a near optimal neighborhood of $\tau^0.$ Our main contribution in Section \ref{sec:nuisance} of the main article is to develop a mathematical theory that supports this phenomenon and also to refine this process to allow for detection of absence of the change point. In the following points we show how the initializing Condition D assumed in Section \ref{sec:nuisance}, requires nothing more than the above described behavior, i.e., any initial value $\check\tau$ separated from the boundaries of the parametric space of $\tau^0,$ and carrying any small or marginal information on $\tau^0$ satisfies this condition.

The main restriction in this condition is that (simplified here for clearer exposition, see Condition D in Section \ref{sec:nuisance} for details),
\begin{eqnarray}\label{eq:initcond1}
|\check\tau-\tau^0|\le c_u\Big(\frac{1}{T}\Big)^{1/k},\quad{\rm{for\,\,any}},\,\, k\in[1,\iny),\,\,{\rm and}\,\,c_u>0.
\end{eqnarray}
Note that the constant $k\in[1,\iny)$ may be arbitrarily large.\footnote{$k\in [1,\iny)$ is arbitrarily large as long as Condition A1 in the manuscript is satisfied. This ensures the `fractional information' in $\check\tau$ is not dominated by the noise terms in the analysis. If $s$ is bounded above, then $k$ is truly arbitrary.}. This means that, if we first pick any $\check\tau\in (0,1),$ separated from its boundaries, then, the farther the user chosen $\check\tau$ is from the true change point $\tau^0,$ the larger the value of $k$ would be, in order to satisfy this initial condition. Furthermore, choosing $c_u=1,$ if we consider the disallowed case of $k=\iny,$ then for any $\tau^0,$ and any initial $\check\tau\in(0,1],$ the initial condition is trivially satisfied since the condition (\ref{eq:initcond1}) requires $|\check\tau-\tau^0|\le 1,$ at $k=\iny.$ This also implies that, if we pick virtually any initial $\check\tau\in (0,1)$ away from its boundaries, then it will satisfy the required initial condition for some large enough $k\in[1,\iny).$ In summary, all that the initial condition requires is the existence of some finite $k<\iny,$ thereby illustrating that this initial condition is infact very mild.

The main novelty of Theorem \ref{thm:alg1cp} is to show that, irrespective of the value of $k$ in the initializing condition, the updated change point estimate $\h\tau$ of Algorithm 1 will satisfy near optimal error bounds, i.e, $|\h\tau-\tau_0|=O(s\log p/T),$ under mild conditions. Importantly, note that error bound is free of $k.$ To see the equivalence of this result with the observation from Figure \ref{fig:insensitivity1}, note that, if we pick any two distinct initializers $\check\tau_1$ and $\check\tau_2,$ where first initial value is closer to the truth $\tau^0,$ i.e., for some $1\le k_1<k_2<\iny,$ then, the corresponding updated change point estimates $\h\tau_1,$ and $\h\tau_2$ will both be in a near optimal neighborhood of $\tau^0.$ This basically implies that the quality of the guess does not influence the updated estimate in its eventual rate of convergence. This is exactly the surprising behavior observed in Figure \ref{fig:insensitivity1}. Furthermore, this also brings out the powerfulness of the proposed Algorithm 1 which is that there is not just one theoretically valid initializer, instead our results show that all values of the initializer in nearly the entire parametric space of $\tau^0,$ are equally theoretically valid initializers.

To conclude this discussion on Condition D, in the following we explicitly describe the above observed property in a large class of problems. Consider the high dimensional model (\ref{model:subgseries}) where $\log p=o(T^{\delta}),$ for some $0<\delta<1,$ the sparsity parameter diverges at a sufficiently slow rate, $s=o(\log T),$ and the change point $\tau^0$ is such that it satisfies, $\tau^0\wedge(1-\tau^0)\ge 1/\log T,$ i.e. it does not converge to zero too fast. Now choose any constant $0<c_1<0.5,$ then our results state that any value of the initializer $\check\tau\in(c_1,1-c_1)$ will be a theoretically valid choice, in the sense that for $T$ large enough (how large a $T$ is required will depend on the choice of $c_1$), the updated $\h\tau$ will satisfy the localization error bound of Theorem \ref{thm:alg1cp}. This can be confirmed by noting that for any $\check\tau\in(c_1,1-c_1)$ will satisfy Condition D of the article for some large enough constant $k>0,$ furthermore any such constant $k$ will in turn satisfy the rate condition of Condition A1 of the article, thereby allowing applicability of our results. For the general case with $s\log p/T\to 0,$ $\tau\ge l_n,$ for some positive sequence $l_n$ where the rate of divergence of $s$ and the rate of convergence of $l_n$ are potentially faster than those assumed earlier. An explicit rule to choose a theoretically valid $\check\tau$ cannot be provided, since all these rates together with the boundaries of the range of theoretically valid initializers shall be inter-related. Consequently, we have stated this inter-relationship between these rates by the means of Condition D and Condition A1 in the manuscript. Although, it is quite apparent, that even in this case the range of theoretically valid initializers will be almost the entire parametric space of $\tau.$ Furthermore, this conclusion is empirically verified in our manuscript with  numerical simulations, where no prior knowledge of $\tau^0$ is assumed (and we consider several cases of $\tau$ ranging from $0.2$ to $0.8$) and the initializer is simply chosen as $\check\tau= 0.5,$ which is the worst possible choice of the initializer assuming no prior information on the unknown change point $\tau^0.$

\subsection*{Numerical results: additional results and omitted details}

\textbf{\emph{Choice of tuning parameters}}: The regularizers $\la_1,$ $\la_2$ used to obtain soft thresholded mean estimates and the regularizer $\gamma$ of Step 1 of Algorithm 1 are all chosen via a BIC type criteria suggested in \cite{kim2012consistent}, which we modify for the model under consideration. Specifically we set $\la_1=\la_2=\la,$ and evaluate $\h\mu_1(\la),$ and $\h\mu_2(\la)$ over an equally spaced grid of $50$ values in the interval $(0,0.5).$ Upon letting $\h S=\{j\,\,\h\mu_{1j}\ne 0\}\cup \{j\,\,\h\mu_{2j}\ne 0\}$ we evaluate the criteria,
\benr
BIC(\la_1,\tau)= \sum_{t=1}^{\lfloor T\tau^0\rfloor}\|y_t-\h\mu_1(\la)\|_2^2+\sum_{t=\lfloor T\tau\rfloor+1}^{T}\|y_t-\h\mu_2(\la)\|_2^2+ |\h S|\log T.\nn
\eenr
For Step 0 of Algorithm 1, we choose that value of $\la$ that minimizes $BIC(\la,\check\tau).$ On the other hand, to obtain the final nuisance mean estimates of Corollary \ref{cor:final}, we choose that value of $\la$ that minimizes $BIC(\la,\h\tau),$ where $\h\tau$ is the change point estimate obtained from Step 1 of Algorithm 1.

The regularizer $\gamma$ of Step 1 of Algorithm 1 is evaluated quite analogously as above. Specifically, we obtain $\h\tau(\gamma),$ for each value of $\gamma$ in a equally spaced grid of $50$ values between $(0,1)$ and compute the criteria,
\benr
BIC(\gamma)= \sum_{t=1}^{\lfloor T\tau^0\rfloor}\|y_t-\h\mu_1\|_2^2+\sum_{t=\lfloor T\tau\rfloor+1}^{T}\|y_t-\h\mu_2\|_2^2+ \big(|\h S|+\|\h\tau(\gamma)\|_0^*\big)\log T.\nn
\eenr
Here $\h\mu_1$ and $\h\mu_2$ represent mean estimates obtained on the binary partition yielded by $\h\tau(\gamma).$  Finally, we choose that value of $\gamma$ that minimizes $BIC(\gamma).$

\textbf{\emph{Computation of $\h\si^2$ and $\h\xi^2$}}: Here we discuss the computation of $\h\si^2$ and $\h\xi^2$ utilized for the computation of confidence intervals for $\tau^0$ using the result of Theorem \ref{thm:limitingdist}. First note that the proposed inference methodology $PLS,$ is implemented in conjunction with the $AL1$ method utilized for preliminary nuisance estimates, accordingly let $\h\mu_1,$ $\h\mu_2$ be the estimates described in Corollary \ref{cor:final}. Additionally let $\h\theta_1$ and $\h\theta_2$ are as defined in Section \ref{sec:intro}. Now recall that by definition,
$\xi=\|\mu_1^0-\mu_2^0\|_2^2=\theta_1^0-\theta_2^0,$ and $\si^2=\lim_T \eta^{0T}\Si_{\vep}\eta^0\big/\xi^2.$ Furthermore note that ${\rm var(\psi_t)}=\eta^{0T}\Si_{\vep}\eta^0,$ where $\psi_t$ are the noise variables of model (\ref{mod:projectedseries}). Accordingly, we can obtain plugin estimates  $\h\xi^2=\h\theta_1-\h\theta_2,$ and
\benr
\h\si^2\big(\tilde\tau,\h\mu_1,\h\mu_2\big)= \frac{1}{\h\xi^2T}\Bigg\{\sum_{t=1}^{\lfloor T\tilde\tau\rfloor}(\h z_t-\h\theta_1)^2+\sum_{t=\lfloor T\tilde \tau\rfloor+1}^{T}(\h z_t-\h\theta_2)^2\Bigg\}.\nn
\eenr
Although these estimates are expected to be consistent, however shrinkage biases present in the mean estimates $\h\mu_1,$ and $\h\mu_2$ seep into the estimation of the variance and jump size leading to significant deviations from significance levels in the simulations. To alleviate these shrinkage biases, we utilize the well accepted and well understood methodology of using refitted parameter estimates, see, e.g. \cite{belloni2017pivotal}. Specifically, instead of using $\h\mu_1$ and $\h\mu_2$ for the variance and jump size calculations, we use their refitted versions, i.e., $\tilde\mu_1=\big[\bar y_{(0:\tilde\tau]}\big]_{\h S1}$ and $\tilde\mu_2=\big[\bar y_{(\tilde\tau:1]}\big]_{\h S2},$ where $\tilde\tau$ is the $PLS$ estimate of $\tau^0,$ and $\h S1=\{j\,\,\h\mu_{1j}\ne 0\},$ $\h S2=\{j\,\,\h\mu_{2j}\ne 0\}.$

\begin{table}
		 \resizebox{1\textwidth}{!}{
		\begin{tabular}{cccccccc}
			\hline
			\multicolumn{2}{c}{$\tau^0=0.4,$ $s=5$} & \multicolumn{2}{c}{$AL1$}                                     & \multicolumn{2}{c}{$PLS$}                                     & \multicolumn{2}{c}{$WS$}                                      \\ \hline
			$T$                & $p$                & \textbf{bias ($\times 10^2$)} & \textbf{RMSE ($\times 10^2$)} & \textbf{bias ($\times 10^2$)} & \textbf{RMSE ($\times 10^2$)} & \textbf{bias ($\times 10^2$)} & \textbf{RMSE ($\times 10^2$)} \\ \hline
			100                & 50                 & 0.160                         & 1.049                         & 0.124                         & 0.020                         & 1.020                         & 0.004                         \\
			100                & 500                & 0.230                         & 1.015                         & 0.424                         & 0.100                         & 0.990                         & 0.003                         \\
			100                & 750                & 0.180                         & 1.122                         & 0.596                         & 0.130                         & 1.118                         & 0.004                         \\ \hline
			225                & 50                 & 0.178                         & 0.671                         & 0.422                         & 0.040                         & 0.655                         & 0.006                         \\
			225                & 500                & 0.218                         & 0.655                         & 1.826                         & 0.156                         & 0.624                         & 0.008                         \\
			225                & 750                & 0.196                         & 0.674                         & 2.655                         & 0.089                         & 0.613                         & 0.009                         \\ \hline
			350                & 50                 & 0.060                         & 0.247                         & 0.821                         & 0.037                         & 0.230                         & 0.008                         \\
			350                & 500                & 0.023                         & 0.214                         & 4.130                         & 0.017                         & 0.218                         & 0.008                         \\
			350                & 750                & 0.046                         & 0.343                         & 6.182                         & 0.017                         & 0.323                         & 0.010                         \\ \hline
	\end{tabular}}
	\caption{\footnotesize{Results of Simulation A.I: estimation performance of $AL1,$ $PLS$ and $WS$ methods. Here, bias ($|E(\h\tau-\tau^0)|$), and root mean squared error (RMSE, $\big\{E(\h\tau-\tau^0)^2\big\}^{1/2}$)}}
	\label{tab:estimation.app1}
\end{table}

\begin{table}
		 \resizebox{1\textwidth}{!}{
		\begin{tabular}{cccccccc}
			\hline
			\multicolumn{2}{c}{$\tau^0=0.6,$ $s=5$} & \multicolumn{2}{c}{$AL1$}                   & \multicolumn{2}{c}{$PLS$}                   & \multicolumn{2}{c}{$WS$}                    \\ \hline
			$T$                & $p$                & bias ($\times 10^2$) & RMSE ($\times 10^2$) & bias ($\times 10^2$) & RMSE ($\times 10^2$) & bias ($\times 10^2$) & RMSE ($\times 10^2$) \\ \hline
			100                & 50                 & 0.270                & 1.054                & 0.127                & 0.090                & 1.034                & 0.004                \\
			100                & 500                & 0.010                & 0.671                & 0.432                & 0.060                & 0.600                & 0.005                \\
			100                & 750                & 0.140                & 1.030                & 0.593                & 0.050                & 1.005                & 0.003                \\ \hline
			225                & 50                 & 0.102                & 0.395                & 0.389                & 0.031                & 0.298                & 0.006                \\
			225                & 500                & 0.004                & 0.317                & 1.868                & 0.013                & 0.324                & 0.008                \\
			225                & 750                & 0.044                & 0.586                & 2.607                & 0.004                & 0.468                & 0.006                \\ \hline
			350                & 50                 & 0.069                & 0.283                & 0.806                & 0.011                & 0.194                & 0.010                \\
			350                & 500                & 0.034                & 0.218                & 4.297                & 0.014                & 0.212                & 0.011                \\
			350                & 750                & 0.054                & 0.304                & 5.794                & 0.011                & 0.277                & 0.011                \\ \hline
	\end{tabular}}
	\caption{\footnotesize{Results of Simulation A.I: estimation performance of $AL1,$ $PLS$ and $WS$ methods. Here, bias ($|E(\h\tau-\tau^0)|$), and root mean squared error (RMSE, $\big\{E(\h\tau-\tau^0)^2\big\}^{1/2}$)}}
	\label{tab:estimation.app2}
\end{table}

\begin{table}
		 \resizebox{1\textwidth}{!}{
		\begin{tabular}{cccccccc}
			\hline
			\multicolumn{2}{c}{$\tau^0=0.8,$ $s=5$} & \multicolumn{2}{c}{$AL1$}                   & \multicolumn{2}{c}{$PLS$}                   & \multicolumn{2}{c}{$WS$}                    \\ \hline
			$T$                & $p$                & bias ($\times 10^2$) & RMSE ($\times 10^2$) & bias ($\times 10^2$) & RMSE ($\times 10^2$) & bias ($\times 10^2$) & RMSE ($\times 10^2$) \\ \hline
			100                & 50                 & 1.750                & 3.637                & 0.125                & 0.140                & 0.849                & 0.006                \\
			100                & 500                & 0.910                & 3.500                & 0.419                & 0.580                & 3.206                & 0.003                \\
			100                & 750                & 0.140                & 1.985                & 0.587                & 0.260                & 1.140                & 0.004                \\ \hline
			225                & 50                 & 0.827                & 1.616                & 0.388                & 0.093                & 0.419                & 0.006                \\
			225                & 500                & 0.329                & 0.871                & 1.845                & 0.058                & 0.481                & 0.006                \\
			225                & 750                & 0.356                & 0.982                & 2.625                & 0.089                & 0.586                & 0.007                \\ \hline
			350                & 50                 & 0.480                & 1.022                & 0.845                & 0.011                & 0.277                & 0.012                \\
			350                & 500                & 0.297                & 0.652                & 3.995                & 0.054                & 0.312                & 0.010                \\
			350                & 750                & 0.343                & 0.668                & 5.999                & 0.009                & 0.174                & 0.012                \\ \hline
	\end{tabular}}
	\caption{\footnotesize{Results of Simulation A.I: estimation performance of $AL1,$ $PLS$ and $WS$ methods. Here, bias ($|E(\h\tau-\tau^0)|$), and root mean squared error (RMSE, $\big\{E(\h\tau-\tau^0)^2\big\}^{1/2}$)}}
	\label{tab:estimation.app3}
\end{table}

\bibliographystyle{plainnat}
\bibliography{meanchange}

\end{document}